\documentclass[11pt]{amsart}
\usepackage{amsmath, amsthm, amssymb, epsfig, amsfonts, bbm, euscript, graphicx, url}
\usepackage{pb-diagram}
\usepackage{lamsarrow}
\usepackage[colorlinks]{hyperref}

\newtheorem{theorem}{Theorem}[subsection]
\newtheorem{lemma}[theorem]{Lemma}
\newtheorem{proposition}[theorem]{Proposition}
\newtheorem{corollary}[theorem]{Corollary}

\newtheorem{definition}[theorem]{Definition}

\newtheorem{conjecture}[theorem]{Conjecture}

\theoremstyle{definition}
\newtheorem{step}{Step}
\newtheorem{example}[theorem]{Example}

\newtheorem{remark}[theorem]{Remark}

\numberwithin{equation}{section}

\newskip\aline \newskip\halfaline
\DeclareMathOperator{\GL}{GL}

\DeclareMathOperator{\Int}{Int}
\DeclareMathOperator{\id}{id}

\newcommand{\bC}{\mathbb C}

\newcommand{\bR}{\mathbb R}

\newcommand{\bZ}{\mathbb Z}

\title{Hamiltonian $S^1$ actions with Isolated Fixed Points on $6$-Dimensional Symplectic Manifolds}
\author{Andrew Fanoe}
\begin{document}
\maketitle
\begin{abstract}The question of what conditions guarantee that a symplectic $S^1$ action is Hamiltonian has been studied for many years.  In \cite{TW} Sue Tolman and Jonathon Weitsman proved that if the action is semifree and has a non-empty set of isolated fixed points then the action is Hamiltonian.  Furthermore, in \cite{CHS} Cho, Hwang, and Suh proved in the 6-dimensional case that if we have $b_2^+=1$ at a reduced space at a regular level $\lambda$ of the circle valued moment map, then the action is Hamiltonian.  In this paper, we will use this to prove that certain 6-dimensional symplectic actions which are not semifree and have a non-empty set of isolated fixed points are Hamiltonian.  In this case, the reduced spaces are 4-dimensional symplectic orbifolds, and we will resolve the orbifold singularities and use J-holomorphic curve techniques on the resolutions.
\end{abstract}
\section{Introduction}
\subsection{Statement of Results}
Consider a closed symplectic manifold $(X^6,\omega)$ with a symplectic $S^1$ action.  In this note, we will prove a special case of the following conjecture.  
\begin{conjecture}\label{mainconjecture}
If $(X^6,\omega)$ has a symplectic $S^1$ action which has a non-empty set of isolated fixed points, then the action is Hamiltonian.
\end{conjecture}
In order to prove this, we will consider the case where the action is not Hamiltonian.  In this case, McDuff noticed that if the symplectic class is integral, the circle action is determined by a moment map with values in $S^1$.  By perturbing the symplectic form, we can assume that $[\omega]$ is rational so that a large multiple of $[\omega]$ is integral, so that we can always assume have such a moment map.  Moreover, we also have no critical points of index or co-index $0$, or of odd index or co-index.  In particular, we only have critical points of index $2$ and $4$.

For each $\lambda$ in $S^1$, we can form the reduced space $M_{\lambda}$ by first considering $\Phi^{-1}(\lambda)$, where $\Phi$ is the moment map, and then quotienting this by the $S^1$ action.  The resulting space will in general be a four dimensional symplectic orbifold with orbifold singularities corresponding to non-trivial isotropy of the $S^1$ action.  We assume that all orbifold singularities are isolated points.  We will resolve these singularities by successive blowups, which adds curves of self intersection $-2$ or less.  We denote the resulting space $\widetilde{M}_\lambda$, and we call it the resolution of $M_\lambda$.  For more details on these resolutions, see section $2$.

We will use the following theorem of Cho, Hwang, and Suh from \cite{CHS} to prove a special case of Conjecture \ref{mainconjecture} above.

\begin{theorem}\label{b+=1impliesdone}
Let $(X^6,\omega)$ be a closed symplectic manifold with a symplectic $S^1$ action with non-empty fixed point set.  Then if there exists a regular value $\lambda$ of the $S^1$ moment map such that the reduced space $M_{\lambda}$ satisfies $b_2^+(M_\lambda)=1$, the action is Hamiltonian.
\end{theorem}

\begin{remark}\label{blowupremark}Since the blowup operation has no effect on $b_2^+$, it is sufficient for us to consider the resolutions $\widetilde{M}_\lambda$ instead of $M_\lambda$ in order to show $b_2^+=1$ at some regular level.  In particular, we always have $b_2^+(\widetilde{M}_\lambda)=b_2^+(M_{\lambda})$.
\end{remark}

In this note, we will prove the special case where our $S^1$ actions have isolated fixed points with isotropy weights $(\pm p,\pm q,\mp 1)$, where here we assume that $\gcd(p,q)=1$.

\begin{theorem}\label{maintheorem}
Suppose we have a closed symplectic manifold $(X^6,\omega)$ with a symplectic $S^1$ action with a non-empty set of isolated fixed points, all of whose isotropy weights are either $(\pm p_i,\pm q_i,\mp 1)$, where $p_i>q_i$ and $\gcd(p_i,q_i)=1$, and such that the $S^1$ action has no codimension $2$ isotropy.  Then the $S^1$ action is Hamiltonian. 
\end{theorem}

\begin{remark}
In \cite{G2}, Godinho's main theorem implies the above result if we add the assumptions that for some $p>q$, all the $p_i=p$, all the $q_i=q$, and $p\neq q+1$.  Additionally, Godinho's proof works for more general $(p,q,-r)$ singularities as well, where again all fixed points are assumed to have the same isotropy weights up to sign.  Thus, our main theorem proves a different special case of the conjecture than Godinho's main theorem does.  
\end{remark}

\subsection{Summary of Main Argument}
We now briefly summarize the main points of the argument.  For definitions and further details, see sections $2$ and $3$.  

Let $\lambda_i$ be a critical value of the moment map with isotropy weights $(\pm p_i,\pm q_i,\mp 1)$.  We will show that at $\lambda_i$, the reduced spaces $M_\lambda$ of the $S^1$ action change by a $(p_i,q_i)$-weighted blowup.  We will further show that if we resolve the corresponding orbifold singularities to form $\widetilde{M}_\lambda$, then this blowup produces two chains $Z^1_i$ and $Z^2_j$ of non-generic curves connected by a curve $\widetilde{E}$ which has $\widetilde{E}^2=-1$ and is an exceptional divisor.  In particular, $\widetilde{E}$ has a non-trivial Gromov invariant, and so this curve persists under perturbations.  We then use holomorphic curve techniques on this curve to demonstrate that the reduced spaces must satisfy $b_2^+=1$, so that by Lemma \ref{b+=1impliesdone}, the action is Hamiltonian.

\begin{remark}
We can easily recover the $6$-dimensional case of \cite{TW} where the action is semifree and has isolated fixed points using the above argument.  Namely, in the semifree case, there are no orbifold points and all the blowups are standard smooth blowups.  Thus, in this case we don't have the curves $Z^i_j$ corresponding to the orbifold singularities and all of the curves that appear in blowups and blowdowns are exceptional divisors, which greatly simplifies the $J$-holomorphic curve arguments.  
\end{remark}

\subsection{Acknowledgments}I would like to thank Dusa McDuff for her help in suggesting the topic, suggesting the main approach, and also helping to refine the arguments in the paper through countless revisions.

\section{Definitions and Technical Lemmas}

In this section, we will build up the tools necessary to prove Theorem \ref{maintheorem}.  We begin by giving a general discussion about orbifolds.

\subsection{Orbifolds}

We first give the definition of an orbifold.  To do this, we first define a local uniformizing chart.

\begin{definition}\label{uniformizingchartdefinition}
Let $M^4$ be a topological space, and let $x\in M$ be a point.  Then a \textbf{$C^{\infty}$ local uniformizing chart} at $x$ is a $4$-tuple $(U,\widetilde{U},\Gamma,\phi)$ where $U$ is a neighborhood of $x$ in $M$, $\widetilde{U}\subset\bR^4$, $\Gamma$ is a finite group acting on $\widetilde{U}$ by diffeomorphisms, and $\phi:\widetilde{U}\rightarrow U$ is a continuous, equivariant map so that $\phi:\widetilde{U}/\Gamma\rightarrow U$ is a homeomorphism.
\end{definition}

Using this, we can now define an orbifold

\begin{definition}\label{orbifolddefinition}
Let $M^4$ be a compact Hausdorff topological space and let $x_i\in M^4$, $i=1,\ldots n$ be points.  Then $M$ is a \textbf{smooth orbifold} if there are $C^{\infty}$ local uniformizing charts $(U_i,\widetilde{U}_i,\Gamma_i,\phi_i)$ at $x_i$ so that $U_i\cap U_j=\emptyset$ if $i\neq j$ and $M^4\setminus\{x_1,\ldots,x_n\}$ is locally Euclidean.  Furthermore, if $(U_i,\widetilde{U}_i,\Gamma_i,\phi_i)$ are such local uniformizing charts, a \textbf{smooth orbifold structure} is given by a finite open cover $\mathcal{C}$ of $M$ by $C^{\infty}$ local uniformizing charts $(U_i,\widetilde{U}_i,\Gamma_i,\phi_i)$ so that if $i>n$, $\Gamma_i$ is the trivial group and so that if $U_i\cap U_j\neq\emptyset$ where $j>n$, then 
\[
\phi_{ij}=\phi_j^{-1}\circ\phi_i:\phi_i^{-1}(U_i\cap U_j)\subset\widetilde{U}_i\longrightarrow\phi_j^{-1}(U_i\cap U_j)\subset\widetilde{U}_j
\]
is a diffeomorphism.  
\end{definition} 

\begin{remark}
One can define differential forms in this context in the usual way by defining them on each local uniformizing chart.  In this fashion, it can be shown that all the usual theory of differential forms, including De Rham cohomology and Poincar\'{e} duality carries over to the smooth orbifold case.  Additionally, one can define a symplectic orbifold in the obvious way.
\end{remark}

This leads to the definition of an orbifold singularity

\begin{definition}\label{orbifoldsingularity}
Let $M^4$ be an orbifold.  A point $x\in M$ will be called an \textbf{orbifold singularity} of \textbf{order $r$} and \textbf{type $(p,q)$} where $\gcd(p,r)=\gcd(q,r)=1$ if there is a local uniformizing chart $(U,\widetilde{U},\bZ_r,\phi)$ near $x$ so that $\bZ_r$ acts on $\widetilde{U}\subset\bR^4=\bC^2$ by 
\[
\xi(z_1,z_2)\mapsto (\xi^p z_1,\xi^q z_2)
\]
where $\xi=e^{\tfrac{2\pi i}{r}}$.  Notice that this action is free away from the origin and has an isolated fixed point at the origin.
\end{definition}

\begin{remark}
The above definitions are much simpler than the general definitions of an orbifold and an orbifold singularity.  By the standard terminology, the above would be considered an isolated orbifold singularity.  In general, it is not necessary to assume that orbifolds are $4$-dimensional or that orbifold singularities are isolated, but in our case, these are the only types of orbifolds we will encounter.
\end{remark}

We now discuss what it means to say that two orbifolds are the same

\begin{definition}\label{orbifolddiffeomorphismdefinition}
Let $M,N$ be smooth orbifolds, where $x_i^M$ are the orbifold points of $M$ and $x_j^N$ are the orbifold points of $N$.  Then if $i,j=1,\ldots n$ and $(U_i^X,\widetilde{U}_i^X,\Gamma_i^X,\phi_i^X)$ is a local uniformizing chart for $x_i^X$ with $X=M,N$, we will say that 
\[
\phi:M\longrightarrow N
\]
is a \textbf{diffeomorphism} if the following conditions are satisfied.
\begin{enumerate}
\item $\phi(x_i^M)=x_i^N$, up to reordering.
\item $\phi:M\setminus\{x_1^M,\ldots,x_n^M\}\rightarrow N\setminus\{x_1^N,\ldots,x_n^N\}$ is a diffeomorphism
\item The local uniformizing charts can be chosen so $\phi(U_i^M)=U_i^N$.  Also, $\phi$ lifts to 
\[
\widetilde{\phi}_i:\widetilde{U}_i^M\longrightarrow\widetilde{U}_i^N
\]
where $\widetilde{\phi}_i$ is an equivariant diffeomorphism.  In other words, there is a group isomorphism $h:\Gamma_i^M\rightarrow\Gamma_i^N$ so that 
\[
\widetilde{\phi}_i(\xi\cdot(z_1,z_2))=h(\xi)\cdot\widetilde{\phi}_i(z_1,z_2)
\]
for $\xi\in\Gamma_i^M$ and $(z_1,z_2)\in\widetilde{U}_i^M$.
\end{enumerate}
\end{definition}

We finish this section by proving a lemma which discusses the extent to which the type of an order $r$ orbifold singularity is preserved under such a diffeomorphism

\begin{theorem}\label{orbifoldtypelemma}
Let $M,M'$ be orbifolds, $\phi:M\rightarrow M'$ a diffeomorphism, and $x,x'$ orbifold singularities of $M,M'$ respectively so $\phi(x)=x'$.  In particular, both $x$ and $x'$ have the same order, which we call $r$.  Then if $x$ is of type $(1,q)$, we must have $x'$ of type $(1,q')$ where $\gcd(q,r)=\gcd(q',r)=1$ and either $q'\equiv\pm q\mod r$ or $qq'\equiv \pm 1\mod r$.  Furthermore, if $\phi$ is orientation preserving, $q'\equiv q\mod r$ or $qq'\equiv 1\mod r$
\end{theorem}
\begin{proof}
As above, if we have such a diffeomorphism $\phi:M\rightarrow M'$ so $\phi(x)=x'$, then there are local uniformizing charts $(U,\widetilde{U},\Gamma,\psi)$ and $(U',\widetilde{U}',\Gamma',\psi')$ at $x,x'$ respectively so that $\phi(U)=U'$ and furthermore, there is an isomorphism $h:\Gamma\rightarrow\Gamma'$ and a lift $\widetilde{\phi}$ of $\phi$ so that
\[
\widetilde{\phi}(\xi\cdot(z_1,z_2))=h(\xi)\cdot\widetilde{\phi}(z_1,z_2)
\]
for $\xi\in\Gamma$ and $(z_1,z_2)\in\widetilde{U}$ where by assumption $\Gamma$ and $\Gamma'$ are copies of $\bZ_r$ acting diagonally with weights $(1,q)$ and $(1,q')$.  

Consider the derivative of $\widetilde{\phi}$ at the origin.  We can identify $T_0\widetilde{U}$ and $T_0\widetilde{U}'$ with copies of $\bC^2$ in the standard way.  Furthermore, the actions $\Gamma$ and $\Gamma'$ on $\widetilde{U}$ and $\widetilde{U}'$ give corresponding actions on $T_0\widetilde{U}$ and $T_0\widetilde{U'}$ under the identification to $\bC^2$.  In particular, we get a real linear diffeomorphism
\[
\psi:=d\widetilde{\phi}:\bC^2\longrightarrow\bC^2
\]
so that
\[
\psi(\xi\cdot(z_1,z_2))=h(\xi)\cdot \psi(z_1,z_2)
\]
for $\xi\in\Gamma$ and $(z_1,z_2)\in\bC^2$.  Additionally, $\psi$ is orientation preserving if and only if $\widetilde{\phi}$ was orientation preserving. Now let $\bC^2=\bR^4$, and let $A$, $A'$ denote the linear symplectomorphisms $\xi\cdot$ and $h(\xi)\cdot$ determined by the actions of $\Gamma$ and $\Gamma'$, where $\xi$ and $h(\xi)$ are generators of $\Gamma$ and $\Gamma'$.  Then $A$ and $A'$ have the matrices
\[
A=\left(\begin{matrix}
R(\theta) & 0 \\ 
0 & R(q\theta)
\end{matrix}\right) \quad
A'=\left(\begin{matrix}
R(k\theta) & 0\\
0 & R(kq'\theta)
\end{matrix}\right),
\]
where $\theta=\tfrac{2\pi}{r}$ and for any real number $\alpha$, $R(\alpha)$ denotes the $2\times 2$ rotation matrix with angle $\alpha$.  Thus, by the above we have a commutative square of real linear maps
\begin{equation*}
\begin{diagram}
\node{\bR^4}\arrow{s,l}{A}\arrow{e,t}{\psi}\node{\bR^4}\arrow{s,r}{A'}\\
\node{\bR^4}\arrow{e,b}{\psi}\node{\bR^4}
\end{diagram}
\end{equation*}
Tensoring with $\bC$ gives a corresponding commutative square of complex linear maps
\begin{equation*}
\begin{diagram}
\node{\bC^4}\arrow{s,l}{A_\bC}\arrow{e,t}{\psi_\bC}\node{\bC^4}\arrow{s,r}{A_\bC'}\\
\node{\bC^4}\arrow{e,b}{\psi_\bC}\node{\bC^4}
\end{diagram}
\end{equation*}

A simple computation then shows that the complex linear transformation $A_\bC$ has eigenvalues $\pm\theta$ and $\pm q\theta$ with corresponding eigenvectors $v^\bC_\lambda$.  Multiplying $v^\bC_\lambda$ by a complex number if necessary, the corresponding real vectors $v_\lambda$ formed by taking the real part of $v^\bC_\lambda$ will form a basis of $\bR^4$.  We denote $v_i$ the basis of $\bR^4$ corresponding to the ordering $v_\theta$, $v_{-\theta}$, $v_{q\theta}$, $v_{-q\theta}$.  Similarly, $A'_\bC$ has eigenvalues $\pm k\theta$ and $\pm kq'\theta$ with corresponding eigenvectors $w^\bC_\lambda$.  As before, taking real parts gives us a corresponding basis of $\bR^4$ denoted $w_i$ corresponding to the ordering $w_{k\theta}$, $w_{-k\theta}$, $w_{kq'\theta}$, and $w_{-kq'_\theta}$.  

Since $\psi_\bC$ is complex linear and fits into our commutative square, $\psi_\bC$ must preserve eigenvectors and eigenvalues.  In particular, $\psi_\bC(v^\bC_\theta)=w^\bC_\lambda$ for some $\lambda=\pm k\theta, \pm kq'\theta$.  In particular, we must have that one of $k\theta$, $-k\theta$, $kq'\theta$, or $-kq'\theta$ equals $\theta$ mod $2\pi$.  However, $r\theta=2\pi$, so if $\theta\equiv\pm k\theta\mod 2\pi$, then $k\equiv\pm 1\mod r$ and thus, we must also have $q\equiv \pm q'\mod r$ from the other eigenvalues.  Correspondingly, if $\theta$ equals $\pm kq'\theta$, then we must have $kq'\equiv\pm 1\mod r$ and we must also have $q\equiv \pm k\mod r$ from the other eigenvalues.  Thus, the only possibilities are $q'\equiv\pm q\mod r$ or $qq'\equiv\pm 1\mod r$, as desired.  Thus, it only remains to show that if $\phi$ is orientation preserving, we have $q'\equiv q\mod r$ or $qq'\equiv 1\mod r$.  

To see this, notice that since we know $\psi_\bC$ preserves eigenvectors and eigenvalues, then up to rescaling there is only a finite number of possibilities for $\psi_\bC$.  Namely, $\theta=\pm k\theta,\pm kq'\theta$ gives $4$ choices, and for each choice, there is a corresponding choice of sign in what happens to the eigenvalues $\pm q\theta$.  Thus, there are $8$ total possibilities for the complex linear map $\psi$.  An easy computation shows that exactly $4$ of the choices for $\psi_\bC$ correspond to an orientation preserving $\psi$ on $\bR^4$, where two of them correspond to $q'\equiv q\mod r$ and the other two correspond to $qq'\equiv 1\mod r$.
\end{proof}

\subsection{Resolutions and Almost Complex Structures}
In this section, we will discuss resolutions of orbifold singularities.  We begin by  giving a nice reinterpretation of a symplectic orbifold in terms of symplectic reduction.

\begin{lemma}\label{randomlemma}Consider the symplectic manifold $\bC^3$ with its standard symplectic structure and consider the standard diagonal circle action with weights $(p,q,-r)$ on $\bC^3$ given by 
\[
\lambda\cdot(z_1,z_2,z_3)=(\lambda^pz_1,\lambda^qz_2,\lambda^{-r}z_3)
\]
where $\lambda=e^{2\pi i\lambda}$.  Notice that this action has a Hamiltonian given by $H=p|z_1|^2+q|z_2|^2-r|z_3|^2$ and define $\overline{\bC}^3_\lambda$ to be $H^{-1}(\lambda)/S^1$.  Then $\overline{\bC}^3_0$ is a symplectic orbifold with an orbifold singularity of order $r$ and type $(p,q)$ at the origin.
\end{lemma}

\begin{proof}
$H^{-1}(0)$ consists of all points $(z_1,z_2,z_3)$ so that $p|z_1|^2+q|z_2|^2-r|z_3|^2=0$.  In particular, $|z_3|^2=\tfrac{p}{r}|z_1|^2+\tfrac{q}{r}|z_2|^2$.  Thus, there is a natural, embedding of $\bC^2$ into $H^{-1}(0)$ given by
\[
(z_1,z_2)\mapsto\left(z_1,z_2,\sqrt{\tfrac{p}{r}|z_1|^2+\tfrac{q}{r}|z_2|^2}\right)
\]
which is smooth away from $(0,0))$.  Furthermore, for any $(z_1,z_2,z_3)$ with $H(z_1,z_2,z_3)=0$, there is a $\lambda\in S^1$ so that $z_3\lambda^{-r}=\sqrt{\tfrac{p}{r}|z_1|^2+\tfrac{q}{r}|z_2|^2}$.  

Thus, we can identify $H^{-1}(0)/S^1$ with the set $(z_1,z_2,\sqrt{\tfrac{p}{r}|z_1|^2+\tfrac{q}{r}|z_2|^2})/\sim$, where 
\begin{eqnarray*}
&\left(z_1,z_2,\sqrt{\tfrac{p}{r}|z_1|^2+\tfrac{q}{r}|z_2|^2}\right)\sim\left(w_1,w_2,\sqrt{\tfrac{p}{r}|z_1|^2+\tfrac{q}{r}|z_2|^2}\right)\Longleftrightarrow\\
&\exists\lambda\in S^1:\left(w_1,w_2,\sqrt{\tfrac{p}{r}|z_1|^2+\tfrac{q}{r}|z_2|^2}\right)=\lambda\cdot\left(z_1,z_2,\sqrt{\tfrac{p}{r}|z_1|^2+\tfrac{q}{r}|z_2|^2}\right).
\end{eqnarray*}
However, this can only occur if $\lambda^{-r}(\tfrac{p}{r}|z_1|^2+\tfrac{q}{r}|z_2|^2)=\tfrac{p}{r}|z_1|^2+\tfrac{q}{r}|z_2|^2$, which in turn implies $\lambda^{-r}=1$, so that $\lambda\in\bZ_r\subset S^1$ generated by $\xi=e^{\tfrac{2\pi i}{r}}$.  

Thus, using our embedding of $\bC^2$ into $H^{-1}(0)$, we can identify $H^{-1}(0)/S^1$ with $\bC^2/\bZ_r$, where $\bZ_r$ acts by $\xi\cdot(z_1,z_2)=(\xi^pz_1,\xi^qz_2)$, as desired.
\end{proof}

We now use this to show that any such isolated orbifold singularity has a local toric structure.

\begin{proposition}\label{orbifoldresolutiontheorem2}
Consider the symplectic manifold $\bC^3$ with its standard symplectic structure and consider the standard diagonal circle action with weights $(p,q,-r)$ on $\bC^3$ given by 
\[
\lambda\cdot(z_1,z_2,z_3)=(\lambda^pz_1,\lambda^qz_2,\lambda^{-r}z_3)
\]
where $\lambda=e^{2\pi i\lambda}$ and let $\overline{\bC}^3_\lambda$ be as before.  Then $\overline{\bC}^3_0$ has a toric structure given by a torus action $\overline{T}$ whose moment polytope is the wedge with conormals $(0,-1)$ and $(-r,q\alpha)$, where $\alpha p+\beta r=1$.  Furthermore, $\overline{\bC}^3_\epsilon$ with $\epsilon>0$ has a toric structure given by $\overline{T}$ whose moment polytope is the wedge with conormals $(0,-1)$, $(-r,q\alpha)$, and $(-p,-q\beta)$ where $\alpha$ and $\beta$ are as before.
\end{proposition}
\begin{proof}
$\overline{\bC}^3_\epsilon$ for all $\epsilon\geq 0$ inherits a torus action $\overline{T}$ by taking the standard torus action $T$ on $\bC^3$ and quotienting by the diagonal $S^1$ action with weights $(p,q,-r)$.  The moment polytope of this toric structure on $\overline{\bC}^3_{\epsilon}$ has an embedding into $\mathfrak{t}^*$ by taking the moment polytope of the standard action, $\bR^3_{\geq 0}\subset \mathfrak{t}^*\cong \bR^3$ and restricting to the plane $px+qy-rz=\epsilon$.  We will call this plane $\mathcal{H}^*(\epsilon)$, and we will denote its integer lattice in $\mathfrak{t}^*$ by $\mathcal{H}^*_{\bZ}(\epsilon)$.  This polytope is the piece of $\mathcal{H}^*(\epsilon)$ which has $x,y,z\geq 0$.  

If $\epsilon=0$, we have $\mathcal{H}^*(0)$ is the plane $px+qy-rz=0$.  If $z=0$, we have $px+qy=0$ which means $x=y=0$.  If $y=0$, we have $px=rz$ which gives the ray starting at $(0,0,0)$ with direction $(r,0,p)$.  If $x=0$, we have $qy=rz$ which gives the ray starting at $(0,0,0)$ with direction $(0,r,q)$.  Our polytope is then clearly given by the wedge between $(r,0,p)$ and $(0,r,q)$.

If $\epsilon>0$, we have $px+qy=rz+\epsilon$.  If $z=0$, we have $px+qy=\epsilon$ which gives the line segment in the direction $(-q,p,0)$ between $(\tfrac{\epsilon}{p},0,0)$ and $(0,\tfrac{\epsilon}{q},0)$.  If $y=0$, we have $px=rz+\epsilon$ which gives the ray in the direction $(r,0,p)$ starting at $(\tfrac{\epsilon}{p},0,0)$.  If $x=0$, we have $qy=rz+\epsilon$ which gives the ray in the direction $(0,r,q)$ starting at $(0,\tfrac{\epsilon}{q},0)$.  We denote this section of $\mathcal{H}^*(\epsilon)$ by $\Delta(\epsilon)$.

Furthermore, the moment polytope of the action of $\overline{T}$ also has an embedding into $\overline{\mathfrak{t}}^*\cong\bR^2$.  We similarly denote by $\overline{\mathfrak{t}}^*_{\bZ}$ the integer lattice of this algebra.

We seek to produce an embedding of $\overline{\mathfrak{t}}^*$ into $\mathcal{H}^*(0)$ so that the wedge between $(1,0)$ and $(r,q\alpha)$ maps to the wedge between $(r,0,p)$ and $(0,r,q)$ and furthermore so that the induced map from $\overline{\mathfrak{t}}^*_{\bZ}$ to $\mathcal{H}_{\bZ}^*(0)$ is an element of $\GL(2,\bZ)$ plus a translation.  Similarly, we want an embedding of $\overline{\mathfrak{t}}^*$ into $\mathcal{H}^*(\epsilon)$ so that the wedge between $(1,0)$ and $(r,q\alpha)$ cut by the direction $(q\beta,-p)$ maps to $\Delta(\epsilon)$ and so that the induced map from $\overline{\mathfrak{t}}^*_{\bZ}$ to $\mathcal{H}_{\bZ}^*(\epsilon)$ is an element of $\GL(2,\bZ)$ plus a translation.

We claim that producing such embeddings would complete the proof.  Indeed, the torus $\overline{T}$ is determined both as $\overline{\mathfrak{t}}/\overline{\mathfrak{t}}_{\bZ}$ and as $\mathcal{H}(\epsilon)/\mathcal{H}_{\bZ}(\epsilon)$. Thus, dualizing the embedding would give an embedding from $\mathcal{H}(\epsilon)$ into $\overline{\mathfrak{t}}$ so that $\mathcal{H}_{\bZ}(\epsilon)$ maps by an element of $\GL(2,\bZ)$ plus a translation onto $\overline{\mathfrak{t}}_{\bZ}$.  In particular, this shows that the same torus action $\overline{T}$ is inducing these two moment polytopes, which then gives the desired result.

To produce such an embedding, we will complete $(p,q,-r)$ to an integer basis.  Since $\gcd(p,r)=1$, there exist integers $\alpha$ and $\beta$ so that $\alpha p+\beta r=1$.  Then we have
\[
\det\left(\begin{matrix}
p & q & -r \\ 
0 & 1 & 0\\ 
\beta & 0 & \alpha
\end{matrix}\right)=\alpha p+\beta r=1
\]
In particular, $(p,q,-r)$, $(0,1,0)$, and $(\beta,0,\alpha)$ is an integer basis of $\bZ^3$.  Using this, we can give a basis of $\mathcal{H}(0)$ by giving vectors $e_1$ and $e_2$ so that $e_1\cdot(0,1,0)=e_1\cdot(p,q,-r)=0$ and $e_2\cdot(p,q,-r)=e_2\cdot(\beta,0,\alpha)=0$.  We choose $e_1=(r,0,p)$ and $e_2=(-q\alpha,1,q\beta)$.  Using this basis, we define a linear embedding $\Phi_0$ from $\bR^2$ to $\mathcal{H}^*(0)$ as follows:
\[
\Phi_0(a,b)=ae_1+be_2=(ar-bq\alpha,b,ap+bq\beta)
\]
By construction, $\Phi_0$ is an element of $\GL(2,\bZ)$.  Now notice that $\Phi_0(1,0)=(r,0,p)$, while 
\[
\Phi_0(q\alpha,r)=(rq\alpha-rq\alpha,r,pq\alpha+rq\beta)=(0,r,q(p\alpha+r\beta))=(0,r,q)
\]
Therefore, since $\Phi_0$ is linear, the wedge between $(1,0)$ and $(q\alpha,r)$ maps to the wedge between $(r,0,p)$ and $(0,r,q)$.  Thus, $M$ has a local toric structure given by the torus action $\overline{T}$ whose moment polytope is given by the wedge in $\bR^2$ with conormals $(0,-1)$ and $(-r,q\alpha)$, as desired.  

Also, notice that $\mathcal{H}^*(\epsilon)$ can be formed from $\mathcal{H}^*(0)$ by the translation $\tau_\epsilon(x,y,z)=(x+\tfrac{\epsilon}{p},y,z)$. Thus, we can form an affine embedding $\Phi_\epsilon$ from $\bR^2$ to $\mathcal{H}^*(\epsilon)$ as $\tau_\epsilon\circ\Phi_0$ to get:
\[
\Phi_\epsilon(a,b)=ae_1+be_2+(\tfrac{\epsilon}{p},0,0)=(ar-bq\alpha+\tfrac{\epsilon}{p},b,ap+bq\beta)
\]
By construction, $\Phi_\epsilon$ is an element of $\GL(2,\bZ)$ plus a translation.  Also, as defined, we have 
\[
\Phi_\epsilon((a,b)+(c,d))=\Phi_\epsilon(a,b)+\Phi_0(c,d)
\]
Furthermore, $\Phi_\epsilon(0,0)=(\tfrac{\epsilon}{p},0,0)$ and 
\begin{align*}
\Phi_\epsilon(-\tfrac{\epsilon\beta}{p},\tfrac{\epsilon}{q})&=(-\tfrac{\epsilon}{p}\beta r-\epsilon\alpha-\tfrac{\epsilon}{p},\tfrac{\epsilon}{q},-\epsilon\beta+\epsilon\beta)\\
&=(-\tfrac{\epsilon}{p}(\beta r+\alpha p)+\tfrac{\epsilon}{p},\tfrac{\epsilon}{q},0)\\
&=(-\tfrac{\epsilon}{p}+\tfrac{\epsilon}{p},\tfrac{\epsilon}{q},0)=(0,\tfrac{\epsilon}{q},0)
\end{align*}
Lastly, we see that
\[
\Phi_0(-q\beta,p)=(-q\beta r-pq\alpha,p,-q\beta p+pq\beta)=(-q(p\alpha+r\beta),p,0)=(-q,p,0)
\]
Combining all this, we clearly see that the polytope with conormals $(0,-1)$, $(-p,-q\beta)$ and $(-r,q\alpha)$ maps to $\Delta(\epsilon)$, as desired.

\end{proof}

We can use the above lemma to give a local toric structure to one of our orbifold singularities
\begin{corollary}\label{orbifoldresolutiontheorem}
Let $M^4$ be a symplectic orbifold with orbifold singularity $x$ of order $r$ and type $(p,q)$.  Then a neighborhood of $x$ has a toric structure with moment polytope determined by the conormals $(0,1)$ and $(r,-k)$ where $p\alpha+r\beta=1$, $k\equiv q\alpha\mod r$ and $1\leq k<r$.
\end{corollary}
\begin{proof}
As in Lemma \ref{randomlemma}, a neighborhood of such an orbifold singularity $x$ can be obtained as the reduced space $\overline{\bC}^3_0$ at level $0$ of the diagonal $S^1$ action with weights $(p,q,-r)$ on $\bC^3$.  Furthermore, Lemma \ref{orbifoldresolutiontheorem2} says that this has a toric structure with moment polytope determined by the conormals $(0,-1)$ and $(-r,q\alpha)$.  Consider the following transformation
\[
A=\left(\begin{matrix}
-1 & 0 \\ 
c & -1
\end{matrix}\right) 
\]
Then $A\cdot(0,-1)=(0,1)$, and $A\cdot(-r,q\alpha)=(r,-q\alpha-rc)$.  There is a unique choice of $c$ so this equals $(r,-k)$ where $k\equiv q\alpha\mod r$ and $1\leq k<r$.  This completes the proof.
\end{proof}
\begin{remark}\label{orbifoldresolution}
We can use the above theorem and the techniques of Fulton in \cite{F} to resolve these singularities as follows.  In \cite{F}, Fulton shows that a resolution of the polytope with outer conormals $(0,1)$ and $(m,-k)$ with $0<k<m$ is given by a string of integers $a_i$ so that
\[
\tfrac{m}{k}=a_1-\tfrac{1}{a_2-\tfrac{1}{\ldots-\tfrac{1}{a_n}}},
\]
Then there is a resolution of this singularity by a series of blowups which produces a chain of classes $Z_i$ so that
\[
Z_i\cdot Z_j=\left\{
\begin{array}{ll}
-a_i \text{ if }i=j\\
1 \text{ if }|i-j|=1\\
0 \text{ else}
\end{array}\right.
\]
Furthermore, Fulton shows there is a unique choice of the $a_i$ so that $a_i\geq 2$ for all $i$.  Hence, using the above theorem, we can apply these techniques with $m=r$ and $k\equiv q\alpha \mod r$ with $0<k<r$ as above to get a resolution of any isolated orbifold singularity of order $r$ and type $(p,q)$.
\end{remark}

We can use the above to give the following definition.

\begin{definition}\label{orbifoldresolutiondefinition}
Let $M^4$ be an orbifold.  Then $M$ has a finite set of isolated orbifold singularities, $p_1,\ldots, p_n$.  As in Remark \ref{orbifoldresolution} above, we can get a symplectic manifold $\widetilde{M}$, called the  \textbf{resolution} of $M$ by resolving each of these singularities separately.
\end{definition}

\begin{remark}\label{sameresolutionremark}
Using the above techniques, it is easy to see when two isolated singularities $x,x'$ of orders $r$ and types $(p,q)$ and $(p',q')$ respectively have the same resolution.  Namely, if $x$ is resolved as above by the integer string $a_i$, $i=1,\ldots,n$, then $x'$ will have the same resolution only if $x'$ is resolved by the same string $a_i$, or by the reversed string $\overline{a}_i$, where $\overline{a}_i=a_{n+1-i}$, $i=1,\ldots,n$.  However, a simple induction shows that if
\[
\tfrac{m}{k}=a_1-\tfrac{1}{a_2-\tfrac{1}{\ldots-\tfrac{1}{a_n}}}
\]
then 
\[
a_n-\tfrac{1}{a_{n-1}-\tfrac{1}{\ldots-\tfrac{1}{a_1}}}=\tfrac{m}{k'}
\]
where $kk'\equiv 1\mod m$.  Thus, $x$ and $x'$ will have the same resolution if and only if either $q\alpha\equiv q'\alpha'\mod r$ or $(q\alpha)(q'\alpha')\equiv 1\mod r$, where $\alpha p+\beta r=1$ and $\alpha'p'+\beta'r'=1$.  In particular, if $p=p'=1$, so that $\alpha=\alpha'=1$, and $1\leq q,q'<r$ then we must have $q=q'$ or $qq'\equiv 1\mod r$
\end{remark}

Combining the above remark with Lemma \ref{orbifoldtypelemma} gives us the following useful lemma

\begin{lemma}\label{diffeomorphismliftinglemma}
Let $(X,\omega)$ be a closed $6$-dimensional manifold with an effective, symplectic $S^1$ action with no codimension $2$ isotropy, and consider the family $M_{\lambda}$ of reduced spaces of this action.  Let $\phi$ be any orientation preserving diffeomorphism
\[
\phi:M_\lambda\longrightarrow M_{\lambda'}.
\]
Then $\phi$ lifts to a diffeomorphism
\[
\widetilde{\phi}:\widetilde{M}_\lambda\longrightarrow\widetilde{M}_{\lambda'}
\]
\end{lemma}
\begin{proof}
The proof is an immediate consequence of Lemma \ref{orbifoldtypelemma} and Remark \ref{sameresolutionremark}
\end{proof}

In the above discussion, we showed that given a symplectic orbifold $(M^4,\omega)$ with a finite number of orbifold singularities, there is a corresponding symplectic manifold $(\widetilde{M}^4,\widetilde{\omega})$ which is obtained from $M$ by successive blowups near the singularities.  Moreover, this implies that in $\widetilde{M}$, there are some homology classes with self intersection $\leq -2$ which are represented by symplectically embedded spheres.  We finish this section by discussing which almost complex structures on $\widetilde{M}$ can be blown down to almost complex structures on $M$.  This discussion is largely based on \cite{D2}

More specifically, if $M^4$ has the singularities $p_1,\ldots,p_n$, there are classes $Z_{i,j}$ in $\widetilde{M}$ which are all represented by symplectically embedded spheres $C_{i,j}$ and which satisfy
\begin{equation*}
Z_{i,j}\cdot Z_{k,l}=\left\{
\begin{array}{ll}
-a_{i,j}\leq -2 &\text{if }i=k,~j=l\\
1 &\text{if } i=k,~|j-l|=1\\
0 &\text{else}
\end{array}\right.
\end{equation*} 
Moreover, near a singularity $p_i$, we can blow up any almost complex structure $J$ which is integrable get an almost complex structure $\widetilde{J}$ defined in a neighborhood of $\cup_jC_{i,j}$ which by definition can be blown down to $J$ in a neighborhood of $p_i$.  

With the above in mind, we can give the following definition.

\begin{definition}\label{almostcomplexstructuredefinition}
Let $(\widetilde{M}^4,\widetilde{\omega})$ be the resolution of a symplectic orbifold with singularities $p_i$ and a corresponding set $\mathcal{Z}=\{Z_1,\ldots,Z_n\}$ of homology classes that are represented by symplectically embedded spheres satisfying 
\begin{equation*}
Z_i\cdot Z_j-\left\{
\begin{array}{ll}
-k_i &\text{if }i=j\\
0\text{ or }1 &\text{if }|i-j|=1\\
0 &\text{else}
\end{array}\right.
\end{equation*}
Then we define $\widetilde{\mathcal{J}}(\mathcal{Z}):=\mathcal{J}(\mathcal{Z},\widetilde{\omega})$ to be the space of all $\widetilde{\omega}$-tame almost complex structures which arise as the blowup of an almost complex structure $J$ which is integrable near each $p_i$.
\end{definition}

\begin{remark}\label{almostcomplexstructureremark}
The set $\widetilde{\mathcal{J}}(\mathcal{Z})$ is defined to be isomorphic to the set $\mathcal{J}(p_1,\ldots,p_n;\omega)$ of $\omega$-tame almost complex structures on $M$ which are integrable near $p_i$.  Namely, each $\widetilde{J}\in\widetilde{\mathcal{J}}(\mathcal{Z})$ corresponds to a unique almost complex structure $J\in\mathcal{J}(p_1,\ldots,p_n;\omega)$ in the sense that $J$ blows up to $\widetilde{J}$.
\end{remark}

\subsection{Weighted Blowups and Blowdowns}

We will next discuss ellipsoid blowups.  We will let $E(q,p)=\{(z_1,z_2)|\tfrac{|z_1|^2}{q}+\tfrac{|z_2|^2}{p}\leq 1\}$, where $p>q$ and $\gcd(p,q)=1$.
\begin{definition}\label{weightedblowupdefinition}
Let $(M^4,\omega)$ be a symplectic manifold, and let $x$ be a point, and let $p,q$ be integers with $\gcd(p,q)=1$.  Then the \textbf{$(p,q)$ weighted blowup of size $\epsilon$ at $x$}, denoted $(\widetilde{M},\widetilde{\omega})$ is given by removing 
\[
\Int(\tfrac{\epsilon}{pq}E(q,p))=\{(z_1,z_2):\tfrac{|z_1|^2}{q}+\tfrac{|z_2|^2}{p}=\tfrac{\epsilon}{pq}\}=\{(z_1,z_2)|p|z_1|^2+q|z_2|^2\leq\epsilon\}
\] 
and collapsing the resulting ellipsoid boundary along the characteristic flow to produce a curve $C^E$ in the class $E$, called the \textbf{$(p,q)$-weighted exceptional divisor}.  The form $\widetilde{\omega}$ can be chosen to be $\omega$ outside of $\tfrac{\epsilon}{pq} E(q,p)$ and to satisfy
\[
\int_{C^E}\widetilde{\omega}=\epsilon
\]
\end{definition}

In general, this procedure will not result in a symplectic manifold, but rather in a symplectic orbifold which has two singularities, one of which has order $p$, and the other of which has order $q$.  Furthermore, the $(p,q)$ weighted exceptional curve $E$ will intersect both of these singularities.  To see this, we can look at $\Int(E(q,p))$ in the toric picture.  Under the standard torus action of $\bC^2$, $E(q,p)$ has the moment polytope $\Delta(q,p)$ given by a triangle determined by the conormals $(-1,0)$, $(0,-1)$, and $(-p,-q)$, which can be transformed to the triangle determined by $(-1,0)$, $(0,-1)$, and $(-q,-p)$.  This obviously has a smooth vertex at $(0,0)$.  Thus, we can give a neighborhood $U$ of $\tfrac{\epsilon}{pq} E(q,p)$ a toric structure so that $\tfrac{\epsilon}{pq} E(q,p)$ maps to the corresponding rescaled triangle and the blowup corresponds to cutting out this triangle.  In the polytope, this removes the smooth vertex $(0,0)$ which corresponded to the point $x$ and replaces it with vertices $(p,0)$ and $(0,q)$ which represent orbifold singularities of orders $q$ and $p$ respectively.

\begin{remark}\label{weightedblowupremarkfulton}
As in Remark \ref{orbifoldresolution}, we can resolve these two singularities using the techniques of Fulton.  The result of this procedure is two families of classes, denoted $Z_i^p$ and $Z_i^q$, each corresponding to resolving one of the singularities.  

We cannot directly apply the techniques in \cite{F} for either vertex, but up to some affine transformations, we can apply the techniques.  At the vertex $(0,q)$, we have the conormals $(-1,0)$ and $(-q,-p)$, which map to the conormals $(0,1)$ and $(p,-(p-q))$ under the transformation 
\[
\left(\begin{matrix}
0 & -1\\
-1 & 1
\end{matrix}\right)
\]
Hence, as in Remark \ref{orbifoldresolution}, we get a chain of spheres $Z_i^p$ where $(Z_i^p)^2=-a_i$ where
\[
\tfrac{p}{p-q}=a_1-\tfrac{1}{a_2-\tfrac{1}{\ldots -\tfrac{1}{a_n}}}
\] 
Additionally, at the vertex $(p,0)$, we have the conormals $(0,-1)$ and $(-q,-p)$, which map to the conormals $(0,1)$ and $(q,-k)$ under the transformation
\[
\left(\begin{matrix}
-1 & 0\\
c & -1
\end{matrix}\right)
\]
where $p-cq=-k$, so that $k\equiv-p\equiv q-p\mod q$, with $1\leq k<q$.  Again, as in Remark \ref{orbifoldresolution}, we get a chain of spheres $Z_i^q$ where $(Z_i)^q)^2=-b_j$ where
\[
\tfrac{q}{k}=b_1-\tfrac{1}{b_2-\tfrac{1}{\ldots-\tfrac{1}{b_n}}}
\]
Defined in this way, $Z_i^p\cdot Z_j^q=0$.  Also, as we will show in the below remark if the weighted exceptional divisor is given by a curve $C^E$ in the class $E$, then the proper transform $\widetilde{C}^E$ in the class $\widetilde{E}$ will be an exceptional divisor in the usual sense if $x$ is not an orbifold point.  Furthermore, $\widetilde{C}^E\cdot Z_1^p=1$, $\widetilde{C}^E\cdot Z_n^q=1$, and $\widetilde{C}^E\cdot Z_i^l=0$ for all other choices of $i,l$.  
\end{remark}

\begin{remark}\label{weightedblowupremarkdusa}
The above procedure produces a symplectic manifold $\widetilde{M}$ which is the resolution of the $(p,q)$-weighted blowup of a symplectic manifold $M$ which is obtained by a sequence of blowups, the first of which is the $(p,q)$ weighted blowup itself.  However, as McDuff shows in section $3$ of \cite{D3}, if $x$ is a smooth point the same manifold $\widetilde{M}$ can be obtained from $M$ by a sequence of standard blowups, the last of which corresponds to the $(p,q)$ weighted blowup.  We will demonstrate the general technique by showing how this works for $E(4,7)$.  

First, we write down a sequence of numbers according to the following rule.  First, we let $q_1=q$ and write down $a_1$ copies of $q_1$, where $a_1q_1\leq p<(a_1+1)q_1$.  Next, we let $q_2=p-a_1q_1$ and write down $a_2$ copies of $q_2$ where $a_2q_2\leq q_1<(a_2+1)q_2$.  We continue this procedure inductively until there is an n so that $a_nq_n=q_{n-1}$.  For $E(4,7)$, this gives us the sequence $4,3,1,1,1$.  We then cut the moment polytope of $\bC^2$ successively $a_1$ times down from the vertical edge, $a_2$ times up from the horizontal edge, $a_3$ times down from the last of the $a_1$ blowups, $a_4$ times up from the last $a_2$ blowup and so on.

In our case, this gives us the cuts $(1,1)$, $(1,2)$, $(2,3)$, $(3,5)$, and $(4,7)$, as in Figure $1$. 
\end{remark}

\begin{figure}[t]
\includegraphics[scale=0.5]{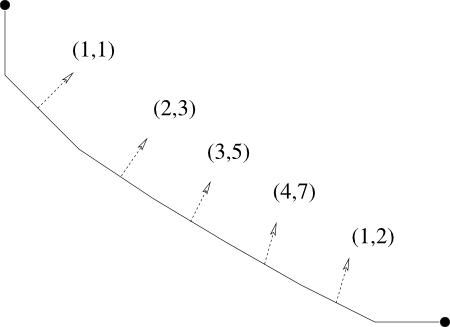}
\caption{Resolution of $(4,7)$ weighted blowup}
\end{figure}

\begin{remark}\label{r>1remark}
The above remarks deal with resolving weighted blowups of smooth points, since as in Lemma \ref{local form of (p,q,-r)}, in our case we never have to consider a weighted blowup of an orbifold singularity.  We proceeded by noting that the weighted blowup has a toric structure with moment polytope determined by the conormals $(-1,0)$, $(0,-1)$, and $(-q,-p)$.  The key point for us was to notice that we could interpret the resolution as arising from a series of smooth blowups of symplectic manifolds, which for example is how we prove Lemma \ref{intersectingweighteddivisorlemma} in $2.5$.  

If $r>1$, then as in Lemma \ref{local form of (p,q,-r)}, we will have a $(p,q)$ weighted blowup of an orbifold singularity.  In this case, we will not have as nice of a toric structure, although the $\epsilon>0$ computation of Lemma \ref{orbifoldresolutiontheorem2} does give us a toric structure with moment polytope determined by the conormals $(0,-1)$, $(-r,q\alpha)$, and $(-p,-q\beta)$ where $p\alpha+r\beta=1$.  Thus, we can resolve the weighted blowup exactly as in Remark \ref{weightedblowupremarkfulton}.

In certain cases, such as a $(5,2)$ weighted blowup at an order $3$ singularity, this resolution will still produce a $-1$ curve, while in other cases, such as a $(3,2)$-weighted blowup at an order $5$ singularity, it does not produce a $-1$ curve.  However, even if the resolution produces a $-1$ curve, the resolution will no longer come from a series of smooth blowups of symplectic manifolds because of the orbifold singularity and thus in this case we cannot prove Lemma \ref{intersectingweighteddivisorlemma}.  We are currently working on resolving this difficulty in order to extend our result to as many other cases as possible, but we have not solved this problem yet which is why we assume $r=1$ in this paper.

\end{remark}
We conclude this section by discussing weighted blowdowns of weighted exceptional divisors.  To start off, we first must discuss exactly what we mean by a weighted exceptional divisor.  

\begin{definition}\label{weightedexceptionaldivisordefinition}
Let $M^4$ be a symplectic orbifold with singularities $x_p$ and $x_q$ of orders $p$ and $q$ respectively.  As in Remark \ref{orbifoldresolution}, this gives us classes $Z_i^p$ and $Z_j^q$ so for $l,m=p,q$,
\begin{equation*}
Z_i^l\cdot Z_j^m=\left\{
\begin{array}{ll}
-k_i^l &\text{if }i=j,~l=m\\
1 &\text{if }l=m~|i-j|=1\\
0 &\text{else}
\end{array}\right.
\end{equation*}
Then we will say that a curve $C^E$ in $M$ through $x_p$ and $x_q$ is a \textbf{$(p,q)$-weighted exceptional divisor} if the class $[C^E]$ in $M$ lifts to a class $\widetilde{E}$ in $\widetilde{M}$ satisfying
\begin{enumerate}
\item $\widetilde{E}$ is an exceptional divisor in the usual sense represented by $\widetilde{C}^E$
\item Up to reversing the order of $i,j$ in $Z_i^p$ and $Z_j^q$, 
\begin{equation*}
\widetilde{C}^E\cdot Z_i^k=\left\{
\begin{array}{ll}
1 &\text{if }i=1\\
0 &\text{else}
\end{array}\right.
\end{equation*}
\end{enumerate}
\end{definition}

\begin{remark}\label{weightedexceptionalblowdownremark}
Given a weighted exceptional divisor $C^E$ as above, Remark \ref{weightedblowupremarkdusa} implies that we can successively blow down $\widetilde{C}^E$, $Z_i^p$ and $Z_j^q$ in $\widetilde{M}$ by smooth blowdowns of exceptional divisors to obtain a manifold $\widehat{M}$ which we call the $(p,q)$ weighted blowdown of $\widetilde{M}$, or just the $(p,q)$ weighted blowdown of $M$.
\end{remark}

We now say more about almost complex structures.  Specifically, we want to discuss which almost complex structures on $M$ or $\widetilde{M}$ can be blown down in the above sense.  The following theorem is based on Theorem $1.2.5$ and Remark $1.2.6$ of \cite{D2}, and describes a certain set of almost complex structures.

\begin{theorem}\label{dusatheorem}
Let $(M^4,\omega)$ be a symplectic orbifold with singularities at $p_1,\ldots, p_n$, and let $(\widetilde{M}^4,\widetilde{\omega})$ be its resolution.  In particular, we have homology classes $Z=\{Z_1,\ldots,Z_n\}$ on $\widetilde{M}$ so that
\begin{equation*}
Z_i\cdot Z_j=\left\{
\begin{array}{ll}
-k_i\leq -2 &\text{if }i=j\\
0\text{ or }1 &\text{if }|i-j|=1\\
0 &\text{else}
\end{array}\right.
\end{equation*}
Let $\widetilde{\mathcal{J}}(\mathcal{Z},\widetilde{\omega})$ and $\mathcal{J}(p_1,\ldots,p_n;\omega)$ be defined as in Definition \ref{almostcomplexstructuredefinition} and Remark \ref{almostcomplexstructureremark}.  Also, let $\widetilde{\mathcal{A}}$ be a finite, disjoint subset of $\widetilde{\mathcal{E}}\subset H_2(\widetilde{M};\bZ)$, the collection of all standard exceptional classes on $\widetilde{M}$.  Further assume for $\widetilde{A}\in\widetilde{\mathcal{A}}$, $\widetilde{A}\cdot Z_i\geq 0$ for all $i$, and $\widetilde{A}\cdot \widetilde{E}\geq 0$ for all $\widetilde{E}\in\widetilde{\mathcal{E}}\setminus\{\widetilde{A}\}$.  

Then, under these assumptions,  there is a subset $\widetilde{\mathcal{J}}(\mathcal{Z};\widetilde{\mathcal{A}})$ of $\widetilde{\mathcal{J}}(\mathcal{Z},\widetilde{\omega})$ which is path connected and residual in the sense of Baire so that for all $\widetilde{J}\in\widetilde{\mathcal{J}}(\mathcal{Z};\widetilde{\mathcal{A}})$, all the classes $\widetilde{A}$ and $Z_i$ are represented by embedded, $\widetilde{J}$-holomorphic spheres so that all intersections are positive and transverse, and a corresponding subset $\mathcal{J}(p_1,\ldots,p_n;\mathcal{A})$ of $\mathcal{J}(p_1,\ldots,p_n;\omega)$ which is also path connected and residual in the sense of Baire.
\end{theorem}

\begin{remark}\label{almostcomplexstructureweightedblowdownremark} Now, let $M^4$ is a symplectic orbifold with singularities at $x_p$ and $x_q$ and let $C^E$ be a $(p,q)$-weighted exceptional divisor through $x_p$ and $x_q$, as defined in Definition \ref{weightedexceptionaldivisordefinition}.  Then, in particular, we have classes $Z_1,\ldots, Z_n$ obtained from resolving $x_p$ and $x_q$, as well as an exceptional divisor $\widetilde{E}$ so that $\widetilde{E}\cdot Z_i\geq 0$ for all $i$.  Then, given any $\widetilde{J}\in\widetilde{\mathcal{J}}(\mathcal{Z},\widetilde{E})$, there is a corresponding $J\in\mathcal{J}(x_p,x_q;E)$, and furthermore, there is also an almost complex structure $\widehat{J}$ on $\widehat{M}$, the $(p,q)$-weighted blowdown of $C^E$.  In other words, any such $\widetilde{J}$ and $J$ can be blown down in the $(p,q)$-weighted sense described above.
\end{remark}

\subsection{Topology of Reduced Spaces}

Now, let $(X,\omega)$ be a closed, $6$-dimensional symplectic manifold with a symplectic $S^1$ action.  We will consider the resulting reduced spaces, which form a family of closed symplectic orbifolds $M_{\lambda}$.  In particular, as we move $\lambda$ counterclockwise around the circle, we will examine how the topology of $M_{\lambda}$ changes.  Recall that we are focusing on the case where the $S^1$ action has isolated fixed points with isotropy weights $(p,q,-1)$ or $(-p,-q,1)$.

The below theorem shows how the reduced spaces change as we move through a critical level.  The statement and proof are based on Theorem 6.1 of \cite{G}

\begin{lemma}\label{local form of (p,q,-r)}Let $(X,\omega)$ be a closed symplectic manifold with a symplectic $S^1$ action which has an isolated fixed point at $x_0\in X$ with isotropy weights $(p,q,-1)$ at the moment map level $\lambda_0$ with $\gcd(p,q)=1$.  Then $M_{\lambda_0+\epsilon}$ is the $(p,q)$ weighted blowup of size $\tfrac{\epsilon}{pq}$ of $M_{\lambda_0}$ at $\overline{x}_0\in M_{\lambda_0}$.
\end{lemma}

\begin{proof}Since the $S^1$ action has isolated fixed points, there is a neighborhood of the $(p,q,-1)$ fixed point which maps equivariantly to $\bC^3$ with the action 
\[
e^{2\pi i\lambda}\cdot(z_1,z_2,z_3)=(z_1e^{2\pi pi\lambda},z_2e^{2\pi qi\lambda},z_3e^{-2\pi i\lambda})
\]
A moment map for this action is given by
\[
H=p|z_1|^2+q|z_2|^2-|z_3|^2
\]
Clearly, 0 is the only critical value of the moment map.  For $\epsilon>0$, we examine the structure of the reduced spaces $\overline{\bC}^3_{-\epsilon}$ and $\overline{\bC}^3_{\epsilon}$.  

First, consider $\overline{\bC}^3_{-\epsilon}$ for $\epsilon\geq 0$.  For $\epsilon=0$, Lemma \ref{randomlemma} gives that $\overline{\bC}^3_0\cong\bC^2$.  For $\epsilon>0$, the same argument works by using the embedding of $\bC^2$ into $H^{-1}(-\epsilon)$ given by
\[
(z_1,z_2)\mapsto\left(z_1,z_2,\sqrt{p|z_1|^2+q|z_2|^2+\epsilon}\right)
\]

Now, consider $\overline{\bC}^3_{\epsilon}$.  Recall from Lemma \ref{randomlemma} that there is a moment map for our $S^1$ action given by $H=p|z_1|^2+q|z_2|^2-|z_3|^2$, and therefore, $\overline{\bC}^3_\epsilon$ can be computed by taking the manifold
\[
|z_3|^2+\epsilon=p|z_1|^2+q|z_2|^2
\]
and quotienting by the $S^1$ action.  Thus, $H^{-1}(\epsilon)$ consists of all points $(z_1,z_2,z_3)$ satisfying
\[
\tfrac{|z_1|^2}{q}+\tfrac{|z_2|^2}{p}=\tfrac{1}{pq}(|z_3^2|+\epsilon)
\]
Reordering terms we see there is an embedding from $\bC^2\setminus(\Int\tfrac{\epsilon}{pq}E(q,p))$ into $H^{-1}(\epsilon)$ defined as follows:
\[
(z_1,z_2)\mapsto\bigl(z_1,z_2,\sqrt{p|z_1|^2+q|z_2|^2-\epsilon}\bigr)
\]
Now, consider $H^{-1}(\epsilon)/S^1$.  As in Lemma \ref{randomlemma} and using the above embedding, we can identify this with the set of points 
\[
\{(z_1,z_2,(p|z_1|^2+q|z_2|^2-\epsilon))\}/\!\sim~\cong~[\bC^2\setminus(\Int\tfrac{\epsilon}{pq}E(q,p))]/\!\sim,
\]
where since $r=1$, $(z_1,z_2)\sim(w_1,w_2)$ if and only if there is $\lambda\in S^1$ so that 
\[
\lambda(p|z_1|^2+q|z_2|^2-\epsilon)=p|w_1|^2+q|w_2|^2-\epsilon.
\]
This gives two cases: either $p|z_1|^2+q|z_2|^2-\epsilon=0$ or $p|z_1|^2+q|z_2|^2-\epsilon>0$.  With respect to our earlier embedding, $p|z_1|^2+q|z_2|^2-\epsilon=0$ corresponds to the boundary of $\bC^2\setminus(\Int\tfrac{\epsilon}{pq}E(q,p))$, while $p|z_1|^2+q|z_2|^2-\epsilon>0$ corresponds to the interior.  

If $p|z_1|^2+q|z_2|^2-\epsilon>0$, then as in Lemma \ref{randomlemma} with $r=1$, we must have $\lambda=1$, so that $(z_1,z_2)=(w_1,w_2)$.  

Now consider $p|z_1|^2+q|z_2|^2-\epsilon=0$.  In this case, any $\lambda\in S^1$ preserves this, since $0$ is a fixed point of the $S^1$ action.  In particular, along this ellipsoid boundary, we collapse the entire $S^1$ action.  However, our $S^1$ action restricted to this ellipsoid boundary is exactly the action which generates the characteristic flow.  Combining this with the above, we see that $\overline{\bC}^3_\epsilon$ is formed from $\overline{\bC}^3_0$ by removing the interior of $\tfrac{\epsilon}{pq}E(q,p)$ and collapsing the boundary along its characteristic flow, which is exactly a $(p,q)$ weighted blowup of size $\epsilon$ at the origin of $\overline{\bC}^3_0$, as claimed.

This all shows that there are neighborhoods $\mathcal{N}(\lambda)$ in $M_\lambda$ so that $\overline{x}_0\in \mathcal{N}(\lambda_0)$ and furthermore, $\mathcal{N}(\lambda_0+\epsilon)$ is the $(p,q)$ weighted blowup of size $\tfrac{\epsilon}{pq}$ of $\mathcal{N}(\lambda_0)$ at $\overline{x}_0\in \mathcal{N}(\lambda_0)$.  In particular, we have a blowup map $\rho_\epsilon$ so that
\[
\rho_\epsilon:\mathcal{N}(\lambda_0)\setminus\{\overline{x}_0\}\longrightarrow\mathcal{N}(\lambda_0+\epsilon)\setminus C_\lambda^E
\]
is a diffeomorphism, where $C_\lambda^E$ is a representative of the weighted exceptional divisor.  Thus, Lemma \ref{homologyinvariantdiffeomorphismliftinglemma} below shows that we can extend $\rho_\epsilon$ to a map
\[
\rho_\epsilon:M_{\lambda_0}\longrightarrow M_{\lambda_0+\epsilon}
\]
so that the restriction
\[
\rho_\epsilon:M_{\lambda_0}\setminus\{\overline{x}_0\}\longrightarrow M_{\lambda_0+\epsilon}\setminus C_\lambda^E
\]
is a diffeomorphism.  The map $\rho_\epsilon$ then clearly identifies $M_{\lambda_0+\epsilon}$ as the $(p,q)$-weighted blowup of $M_{\lambda_0}$ at $\overline{x}_0$, as desired.
\end{proof}

\begin{remark}\label{(-p,-q,r)remark}
An exactly analogous computation for a fixed point with isotropy weights $(-p,-q,1)$ would show that locally, we have $M_{\lambda_0}$ is the $(p,q)$ weighted orbifold blowdown of $M_{\lambda_0-\epsilon}$.  Equivalently we could read the argument backwards to get that $M_{\lambda_0-\epsilon}$ is the $(p,q)$ weighted orbifold blowup of $M_{\lambda_0}$.
\end{remark}

The above discusses how the reduced spaces change when we move across a critical point of the moment map.  The following theorem says that if we move across an interval without critical points, then we do not change the reduced spaces.  This theorem is proven in the introduction of \cite{DH}

\begin{lemma}\label{homologyinvariantlemma}Let $(X,\omega)$ be a $6$-dimensional manifold with an effective, Hamiltonian $S^1$ action with a proper moment map $H$.  Consider the family $M_{\lambda}$ of reduced spaces of this action. Then if $\lambda_0,\lambda_1$ lie inside of an interval of regular values of the moment map, there is an orientation-preserving diffeomorphism 
\[
\phi:M_{\lambda_0}\rightarrow M_{\lambda_1}.
\]
\end{lemma}
Using this, we can prove the following technical diffeomorphism extension lemma which we used in Lemma \ref{local form of (p,q,-r)} and which we will use below.

\begin{lemma}\label{homologyinvariantdiffeomorphismliftinglemma}
Let $(X,\omega)$ be a $6$-dimensional manifold with an effective, Hamiltonian $S^1$ action with isolated fixed points with a moment map $H$ with image $[\lambda-\epsilon,\lambda+\epsilon]$.  Further assume that $\lambda$ is the only interior critical value, which corresponds to a fixed point $x_\lambda\in X$ and further corresponds to a point $p_\lambda$ in the reduced space $M_\lambda$.  

Now, let $\lambda'\neq\lambda$ and assume that we have a neighborhood $\mathcal{N}(t)$ of $M_{\lambda+t}$ for all $t\in[-\epsilon,\epsilon]$ so that $p_\lambda\in\mathcal{N}(0)$ and $\mathcal{N}(t)$ is an open set in $M_{\lambda+t}$.  Furthermore, assume that there is some (possibly empty) closed set $\overline{U}(t)\subset\mathcal{N}(t)$ for $t\in[-\epsilon,\epsilon]$ and diffeomorphisms
\[
\phi_t:\mathcal{N}(\lambda)\setminus{\overline{U}(0)}\longrightarrow\mathcal{N}(t)\setminus\overline{U}(t).
\]
Then shrinking $\mathcal{N}(t)$ if necessary, there are extensions
\[
\psi_t:M_\lambda\setminus{\overline{U}(0)}\longrightarrow M_{\lambda+t}\setminus\overline{U}(t)
\]
so that $\psi_t$ is a diffeomorphism.
\end{lemma}

\begin{proof}
We can define the manifold $X(\lambda,\epsilon)$ by taking
\[
H^{-1}\left(\bigcup_{t\in[\lambda-\epsilon,\lambda+\epsilon]}M_t\setminus\overline{U}(t)\right)
\]
Then $X(\lambda,\epsilon)$ is a compact, symplectic, $6$-dimensional manifold with a Hamiltonian circle action with moment map $H$.  Further, $H$ is proper since $X(\lambda,\epsilon)$ is compact.  Thus, by Lemma \ref{homologyinvariantlemma}, for all $t\in(-\epsilon,\epsilon)$, we know that there is a diffeomorphism
\[
\phi'_t:M_\lambda\setminus\mathcal{N}(\lambda)\longrightarrow M_{\lambda+t}\setminus\mathcal{N}(t)
\]
Extrapolating between $\phi'_t$ and $\phi_t$ gives diffeomorphisms 
\[
\psi_t:M_\lambda\setminus\overline{U}(0)\longrightarrow M_{\lambda+t}\setminus\overline{U}(t)
\]
Possibly shrinking the size of $\mathcal{N}(t)$, we can further assume that $\psi_t$ restricts to $\phi_t$, as desired.
\end{proof}

We can use this diffeomorphism extension lemma to prove the following.

\begin{lemma}\label{homologyinvariantlemma2}
Let $(X,\omega)$ be a closed, $6$-dimensional symplectic manifold with an effective, symplectic $S^1$ action with moment map $\Phi$ and isolated fixed points with isotropy weights $(p,q,-1)$ or $(-p,-q,1)$.  Then if $\lambda_0$ and/or $\lambda_1$ are the only critical values in $[\lambda_0,\lambda_1]$ and $M_{\lambda_0}$ does not differ from $M_{\lambda_1}$ by a weighted blowup as in Lemma \ref{local form of (p,q,-r)}, we have an orientation preserving diffeomorphism
\[
\phi:M_{\lambda_0}\longrightarrow M_{\lambda_1}
\]. 
Furthermore, for any $\lambda,\lambda'\in[\lambda_0,\lambda_1]$ where $[\lambda_0,\lambda_1]$ is an interval as above, there is a diffeomorphism
\[
\widetilde{\phi}:\widetilde{M}_{\lambda}\longrightarrow\widetilde{M}_{\lambda'}
\]
\end{lemma}

\begin{proof}
To prove the first half of the statement, note that if $[\lambda_0,\lambda_1]$ has $\lambda_0$ and/or $\lambda_1$ as the only critical values and the reduced spaces do not differ by weighted blowups as in Lemma \ref{local form of (p,q,-r)}, then as in the proof of Lemma \ref{local form of (p,q,-r)}, there is an $\epsilon$ and neighborhoods $\mathcal{N}(t)$ inside of $M_{\lambda_0+t}$ for $t\in[0,\epsilon]$ such that the fixed point at $\lambda_0$ is in $\mathcal{N}(0)$, and $\mathcal{N}(0)$ is diffeomorphic to $\mathcal{N}(t)$ for all $t$.  Then choosing the closed sets $\overline{U}(t)=\emptyset$ for all $t$, Lemma \ref{homologyinvariantdiffeomorphismliftinglemma} above implies the existence of orientation preserving diffeomorphisms
\[
\phi^0_t:M_{\lambda_0}\longrightarrow M_{\lambda_0+t}
\]
By a similar argument near $\lambda_1$, there is an orientation preserving diffeomorphism
\[
\phi^1_t:M_{\lambda_1-t}\longrightarrow M_{\lambda_1}
\]
Then, by assumption, $[\lambda_0+\epsilon,\lambda_1-\epsilon]$ has no critical values, so by Lemma \ref{homologyinvariantlemma}, there is an orientation-preserving diffeomorphism
\[
\phi':M_{\lambda_0+\epsilon}\longrightarrow M_{\lambda_1-\epsilon}
\] 
Defining $\phi:=\phi^1_\epsilon\circ\phi'\circ\phi^0_\epsilon$, we get the desired diffeomorphism, which finishes the proof of the first statement.

The prove the second statment, we notice that by the above statement and Lemma \ref{homologyinvariantlemma}, there is an orientation preserving diffeomorphism
\[
\phi:M_\lambda\longrightarrow M_{\lambda'}
\]
which then lifts to
\[
\widetilde{\phi}:\widetilde{M}_\lambda\longrightarrow\widetilde{M}_{\lambda'}
\]
by Lemma \ref{diffeomorphismliftinglemma}.
\end{proof}

\subsection{Some Intersection Theory}
In this section, we will prove some useful results pertaining to intersection theory.  First, we will give a useful criterion for determining when a closed, symplectic $4$ manifold has $b_2^+=1$.  We recall Theorem $1.4$ from \cite{D}

\begin{theorem}\label{dusab+=1lemma}
Let $(M,\omega)$ be a closed symplectic $4$-manifold and assume that there exists a symplectically immersed $2$-sphere $C$ with only positively oriented transverse double points.  Then if $c_1(C)\geq 2$, $(M,\omega)$ is rational or ruled.  In particular, $b_2^+=1$.
\end{theorem}
We can use the above theorem to prove the following lemma.

\begin{lemma}\label{b+=1lemma}
Let $(M,\omega)$ be a closed symplectic $4$ manifold.  If $M$ contains two embedded $J$-holomorphic $-1$ spheres $C^E_1$ and $C^E_2$ with $E_1\cdot E_2=k\geq 1$, then $b_2^+=1$.
\end{lemma}
\begin{proof}
To prove this, we resolve exactly $1$ of the intersection points of $C^E_1$ and $C^E_2$ to get a single sphere $C$ in the class $E_1+E_2$ which is immersed with $k-1$ positive transverse double points.  Notice that the immersion points of $C$ come from unresolved intersections of $C^E_1$ and $C^E_2$ which are all positive transverse intersections by positivity of intersections in dimension $4$.  Thus, it remains to show that $c_1(C)\geq 2$.  In fact, we have 
\[
c_1(C)=c_1(E_1+E_2)=c_1(E_1)+c_1(E_2)=1+1=2
\]
as desired.
\end{proof}
We now prove a similar lemma in the case where $(\widetilde{M},\widetilde{\omega})$ is the resolution of an orbifold $(M,\omega)$

\begin{lemma}\label{intersectingweighteddivisorlemma}
Let $\widetilde{M}^4$ be a symplectic manifold which is the resolution of $M^4$, a symplectic orbifold.  Furthermore, let $C^E$ be a weighted exceptional divisor as in Definition \ref{weightedexceptionaldivisordefinition}.  In particular, $\widetilde{M}^4$ has classes $\widetilde{E}$ and $Z_i^j$ for $j=1,2$ represented by curves $\widetilde{C}^E$ and $\widetilde{C}_i^j$ so that $\widetilde{E}$ is the class of an exceptional divisor, $\widetilde{E}\cdot Z_1^j=1$, $\widetilde{E}\cdot Z_l^j=0$ if $l>1$, and 
\[
Z_i^l\cdot Z_j^m=\left\{
\begin{array}{ll}
-a_i^l\leq -2\text{ if }i=j,~l=m\\
1\text{ if }l=m,~|i-j|=1\\
0\text{ else}
\end{array}\right.
\]
Then if $\widetilde{M}$ has an exceptional divisor $\widetilde{C}^{E'}$ in a class $\widetilde{E}'\neq\widetilde{E}$ so that either $\widetilde{E'}\cdot\widetilde{E}\neq 0$ or $\widetilde{E'}\cdot Z_{i_0}^{j_0}\neq 0$ for some $i_0,j_0$, then $b_2^+(M)=1$.
\end{lemma}

\begin{proof}
First, notice that Lemma \ref{b+=1lemma} above implies that $b_2^+(\widetilde{M})=1$ if $\widetilde{E'}\cdot\widetilde{E}\neq 0$.  Since $\widetilde{M}$ differs from $M$ by a sequence of blowups, this implies that $b_2^+(M)=1$ as well.  Thus, we can assume that there is some $(i_0,j_0)$ so that $\widetilde{E'}\cdot Z_{i_0}^{j_0}\neq 0$.

Now recall from Remark \ref{weightedexceptionalblowdownremark} that the collection $\widetilde{C}^E$ and $\widetilde{C}_i^j$ can be successively blown down by smooth blowdowns of exceptional divisors starting with $\widetilde{C}^E$ to form the $(p,q)$-weighted blowdown of the weighted exceptional divisor $C^E$.  

Thus, if we begin performing these successive blowdowns, there will be some intermediate stage where we have a closed symplectic $4$-manifold $\widehat{M}$ so that the proper transform of $C_{i_0}^{j_0}$ to $\widehat{M}$ is an exceptional divisor, which we denote $\overline{C}_{i_0}^{j_0}$.  Then by our assumptions $\widetilde{C}^{E'}$ has a proper transform to a curve $\overline{C}^{E'}$ so that $\overline{C}^{E'}\cap\overline{C}_{i_0}^{j_0}\neq 0$.  Furthermore, if we assume that $(i_0,j_0)$ is the first pair of indices so that this intersection is non-zero, $\overline{C}^{E'}$ will be an exceptional divisor as well.

But then $\widehat{M}$ has two intersecting exceptional divisors, so that by Lemma \ref{b+=1lemma} above, $b_2^+(\widehat{M})=1$.  Now, since $\widehat{M}$ differs form $M$ by a series of blowups and blowdowns, this implies that $b_2^+(M)=1$ as well.
\end{proof}
\section{Proof of Main Theorem}
\subsection{Generalized Bundles}
In this section, we will give a technical definition which will be useful to our proof of Theorem \ref{maintheorem}.

\begin{definition}\label{generalizedbundle}
Let $\{V_\alpha\}_{\alpha\in A}$ be a finite open cover of $S^1$ by intervals so that all triple intersections are empty.  Furthermore, assume that $A$ has a partial ordering so that if $V_{\alpha\beta}:=V_\alpha\cap V_\beta\neq\emptyset$, then either $V_\alpha<V_\beta$ or $V_\beta<V_\alpha$.  Then a \textbf{generalized bundle} over $S^1$ is given by topological spaces $\mathcal{F}_\alpha$ with projections $\pi_\alpha:\mathcal{F}_\alpha\rightarrow S^1$ such that if $V_\alpha\cap V_\beta\neq\emptyset$ and $V_\alpha<V_\beta$, there is a fiberwise inclusion 
\[
\phi_{\alpha\beta}:\pi_\alpha^{-1}(V_{\alpha\beta})\longrightarrow\pi_\beta^{-1}(V_{\alpha\beta}).
\] 
Furthermore, a \textbf{section} of a generalized bundle is a collection of maps $s_\alpha:V_\alpha\rightarrow\mathcal{F}_\alpha$ satisfying $s_\alpha=s_\beta\circ\phi_{\alpha\beta}$ whenever $V_{\alpha\beta}\neq\emptyset$ and $V_{\alpha}<V_{\beta}$.
\end{definition}

\begin{remark}\label{generalizedbundleremark}
This definition differs from the standard definition of a bundle primarily in the fact that the fiber $F_x$ over a point $x\in S^1$ is allowed to change its topological type as we change $x$.  However, a section of a generalized bundle still gives us a notion of a smoothly varying family of elements of the $\mathcal{F}_\alpha$, with one for each $x\in S^1$.  This notion of section is the main reason we gave this definition. 
\end{remark}

\begin{example}\label{reducedspaceexampletrivial}
The family of reduced spaces corresponding to a symplectic $S^1$ action on $(X,\omega)$ gives a trivial example of a generalized bundle.  Namely, we can consider the cover of $S^1$ just given by all of $S^1$, and we can let 
\[
\mathcal{F}_{S^1}:=X/S^1
\]
\end{example}

We will now show how one could put a more complicated reduced bundle structure on the family of reduced spaces.  

\begin{example}\label{reducedspaceexample}
Let $(X^6,\omega)$ be a closed symplectic manifold with a symplectic, non-Hamiltonian $S^1$ action.  Then as before, this has an $S^1$ valued moment map and a family of reduced spaces $M_\lambda$ for $\lambda\in S^1$.  By our earlier assumptions, the fixed point set of this action is a finite set of isolated fixed points, which, perturbing $\omega$ if necessary, we can assume all happen at different moment map levels.  We denote these levels $\lambda_1,\ldots,\lambda_{2n}$.  

We arrange the $\lambda_i$ counterclockwise so that $\lambda_{i}$ has isotropy weights either $(p_i,q_i,-1)$ or $(-p_i,-q_i,1)$ for all $i$, for integers $p_i,q_i$ with $\gcd(p_i,q_i)=1$.  Then define $U_i=(\lambda_i,\lambda_{i+1})$ for $i=1,\ldots 2n-1$, and $U_{2n}=(\lambda_{2n},\lambda_1)$.  Also, define $I_i=(\lambda_i-\epsilon,\lambda_i+\epsilon)$, and assign the partial ordering $I_i<U_i$ for all $i=1,\ldots n$, $I_i<U_{i-1}$ if $i=2,\ldots 2n$, and $I_1<U_{2n}$.  This cover gives the reduced spaces the structure of a generalized bundle.  Indeed, we can define
\[
\mathcal{F}_{U_i}:=\bigcup_{\lambda\in U_i}M_\lambda,\quad\mathcal{F}_{I_j}:=\bigcup_{\lambda\in I_j}M_\lambda
\]
Then Lemma \ref{local form of (p,q,-r)} and Lemma \ref{homologyinvariantlemma} give that the spaces $\mathcal{F}_{U_i}$ and $\mathcal{F}_{I_j}$ are topological spaces which are fibered over $\lambda$ by smooth orbifolds, while on all overlaps they are equal to each other, so that the fiberwise inclusions can just be chosen to be the identity.
\end{example}
\begin{remark}
The generalized bundle $\widetilde{\mathcal{J}}$ that we eventually construct in the below proof will be very similar to the above example.  In particular, it will use the same cover $U_i$ and $I_j$ with the same ordering.  However, $\widetilde{\mathcal{F}}_{U_i}$ and $\widetilde{\mathcal{F}}_{I_j}$ will not be fibered by $M_\lambda$, but rather they will be fibered by carefully chosen spaces of almost complex structures of $\widetilde{M}_\lambda$.
\end{remark}

\subsection{Proof of Main Result}
We will now prove our main result, which we restate here for convenience.

\begin{proposition}
Suppose we have a closed symplectic manifold $(X^6,\omega)$ with a symplectic $S^1$ action with a non-empty set of isolated fixed points, all of whose isotropy weights are either $(p_i,q_i,-1)$ or $(-p_i,-q_i,1)$, where $p_i>q_i$ and $\gcd(p_i,q_i)=1$, and such that the $S^1$ action has no codimension $2$ isotropy.  Then the action is Hamiltonian.
\end{proposition}

\begin{proof}
We will assume that the action is not Hamiltonian and derive a contradiction.  Recall from before that since we have a symplectic circle action which is not Hamiltonian, we can assume we have an $S^1$ valued moment map and that we can form the corresponding reduced spaces $M_{\lambda}$ for $\lambda\in S^1$.  

Furthermore, as in Example \ref{reducedspaceexample} above, our moment map can be assumed to have $n$ critical levels which correspond to the isolated fixed points.  We can arrange them counterclockwise at levels $\lambda_i\in S^1$ so that $\lambda_{i}$ has isotropy weights $(p_i,q_i,-1)$ if $i$ is odd or $(-p_i,-q_i,1)$ if $i$ is even with $\gcd(p_i,q_i)=1$, and we can define the sets $U_i$ and $I_j$ as in Example \ref{reducedspaceexample}.

Also, since we assumed that the original $S^1$ action has no codimension $2$-isotropy, we know that each $M_\lambda$ is a symplectic orbifold with a finite number of isolated orbifold singularities, denoted $p_i^\lambda$. Thus we have $\widetilde{M}_\lambda$, the unique resolution of those singularities as in Definition \ref{orbifoldresolutiondefinition}.  As in Remark \ref{orbifoldresolution}, there are homology classes $Z_{i,j}^\lambda$ coming from the blowups used to resolve the singularities $p_i^\lambda$.  We have 
\begin{equation*}
Z^\lambda_{i,j}\cdot Z^\lambda_{l,m}=\left\{
\begin{array}{ll}
-k_{i,j}^\lambda\leq-2 &\text{if }i=l,~j=m\\
1 &\text{if } i=l,~|j-m|=1\\
0 &\text{else}
\end{array}\right.
\end{equation*}
We will let $\mathcal{Z}_\lambda$ denote the union of all these classes over $i,j$.

By Theorem \ref{b+=1impliesdone}, if for some regular level $\lambda$ we have $b_2^+(M_{\lambda})=1$, the action is Hamiltonian which is a contradiction.  Hence, $b_2^+(M_{\lambda})>1$ for all $\lambda$, and thus also $b_2^+(\widetilde{M}_\lambda)>1$.  

We will use this to derive a contradiction in $4$ steps.  First, we will use the language of generalized bundles to find a preferred family $\widetilde{J}(\lambda)$ of almost complex structures on $\widetilde{M}_\lambda$. 

\begin{step}[Constructing the generalized bundle $\widetilde{\mathcal{J}}$]
Consider $U_k$.  By Lemmas \ref{local form of (p,q,-r)} and \ref{homologyinvariantlemma}, there is a $\lambda'\in U_k$ and a smooth family of orientation-preserving diffeomorphisms 
\[
\phi_\lambda:M_{\lambda'}\longrightarrow M_\lambda.
\]
Recall that we can form $\widetilde{M}_\lambda$ with it corresponding set of homology classes $\mathcal{Z}_\lambda$.  As in Definition \ref{almostcomplexstructuredefinition}, we have the set $\widetilde{\mathcal{J}}(\mathcal{Z}_\lambda)$ so that any $\widetilde{J}\in\widetilde{\mathcal{J}}(\mathcal{Z}_\lambda)$ has that $Z_{i,j}^\lambda$ is represented by an embedded, $\widetilde{J}_\lambda$-holomorphic sphere, denoted $C_{i,j}^\lambda$.  Furthermore, by Lemma \ref{diffeomorphismliftinglemma} the diffeomorphisms $\phi_\lambda$ lift to diffeomorphisms $\widetilde{\phi}_\lambda$, where up to a reordering of the $i$ indices, \[
\phi_\lambda^*(Z^\lambda_{i,j})=Z^{\lambda'}_{i,j}.
\]
Also, since $b_2^+(\widetilde{M}_\lambda)>1$, the set $\widetilde{\mathcal{E}}_\lambda$ of homology classes of exceptional divisors on $\widetilde{M}_\lambda$ is finite, and by Lemma \ref{b+=1lemma}, if $\widetilde{E}\neq\widetilde{E}'\in\widetilde{\mathcal{E}}$, then $\widetilde{E}\cdot\widetilde{E}'=0$.  Consider the finite subset $\widetilde{\mathcal{A}}_\lambda\subset\widetilde{\mathcal{E}}_\lambda$ defined by the property that any $\widetilde{A}\in\widetilde{\mathcal{A}}_\lambda$ satisfies $\widetilde{A}\cdot Z_{i,j}^\lambda\geq 0$.  Then, as in Theorem \ref{dusatheorem}, there is a subset $\widetilde{\mathcal{J}}(\mathcal{Z}_\lambda,\widetilde{\mathcal{A}}_\lambda)\subset\widetilde{\mathcal{J}}(\mathcal{Z}_\lambda)$ which is path connected and residual in the sense of Baire so that for any $\widetilde{J}\in\widetilde{\mathcal{J}}(\mathcal{Z}_\lambda,\widetilde{\mathcal{A}}_\lambda)$, $\widetilde{A}\in\widetilde{A}_\lambda$ is represented by a smooth, embedded $\widetilde{J}$-holomorphic sphere which intersects each curve $C_{i,j}^\lambda$ transversally in $\widetilde{A}\cdot C_{i,j}^\lambda$ distinct points. We define
\[
\widetilde{\mathcal{J}}(\lambda):=\widetilde{\mathcal{J}}(\mathcal{Z}_\lambda,\widetilde{\mathcal{A}}_\lambda)
\]
Also, each $\widetilde{J}\in\widetilde{\mathcal{J}}(\lambda)$ is pulled back from an almost complex structure $J$ on $M_\lambda$, so that we get a corresponding family $\mathcal{J}(\lambda)$ in this fashion.

We further define
\[
\widetilde{\mathcal{F}}_{U_k}:=\bigcup_{\lambda\in U_k}\widetilde{\mathcal{J}}(\lambda)
\] 
We can use the isomorphisms $\widetilde{\phi}_\lambda$ to identify $\widetilde{\mathcal{F}}_{U_k}$ with an open subset of the set of almost complex structures $\widetilde{J}$ so that $\widetilde{J}\in\widetilde{\mathcal{J}}(\lambda',\widetilde{\omega}_t)$ for some smooth path of symplectic forms $\widetilde{\omega}_t$ on $\widetilde{M}_{\lambda'}$, which is a topological space.  Hence, $\widetilde{\mathcal{F}}_{U_k}$ is also a topological space, as desired.

Consider now $I_k=(\lambda_k-\epsilon,\lambda_k+\epsilon)$ for some $\epsilon$.  To define $\widetilde{\mathcal{F}}_{I_k}$, we will first construct an explicit family $\widetilde{J}(\lambda)$ of almost complex structures on $\widetilde{M}_\lambda$ for all $\lambda\in(\lambda_k-2\epsilon,\lambda_k+2\epsilon)$.  By our assumptions, Lemma \ref{local form of (p,q,-r)}, and Remark \ref{(-p,-q,r)remark} we know that if $\lambda_k$ has isotropy weights $(p_k,q_k,-1)$, then $M_{\lambda_k+t}$ is the $(p_k,q_k)$-weighted blowup of $M_{\lambda_k}$ at a smooth point $x^{\lambda_k}$ for all $t\in(0,2\epsilon]$, while if $\lambda_k$ has isotropy weights $(-p_k,-q_k,1)$, the same is true with the signs reversed.  Without loss of generality, we will assume the isotropy weights are $(p_k,q_k,-1)$.  They, by Lemma \ref{homologyinvariantlemma2}, we have orientation-preserving diffeomorphisms
\[
\phi_\lambda:M_{\lambda_k}\rightarrow M_\lambda
\]
for all $\lambda\in(\lambda_k-2\epsilon,\lambda_k]$, where $\phi_{\lambda_k}=\id$.  Furthermore, since $x^{\lambda_k}$ is a smooth point, we have neighborhoods in $M_{\lambda_k}$ denoted $\mathcal{N}(p_i^{\lambda_k})$ and $\mathcal{N}(\lambda_k)$ of the orbifold points $p_i^{\lambda_k}$ and $x^{\lambda_k}$ respectively so that $\mathcal{N}(\lambda_k)\cap\mathcal{N}(p_i^{\lambda_k})=\emptyset$.

Now, as before, consider the resolution $\widetilde{M}_{\lambda_k}$, with its corresponding set of homology classes $\mathcal{Z}_{\lambda_k}$.  Notice that since $x^{\lambda_k}$ stays away from $p_i^{\lambda_k}$, there is a corresponding point $\widetilde{x}^{\lambda_k}\in\widetilde{M}_{\lambda_k}$.  As in Definition \ref{almostcomplexstructuredefinition}, consider the set $\widetilde{\mathcal{J}}(\mathcal{Z}_k,\widetilde{\omega}_{\lambda_k})$ as above.  Recall from Lemma \ref{diffeomorphismliftinglemma} that we can lift the diffeomorphisms $\phi_\lambda$ to diffeomorphisms 
\[
\widetilde{\phi_\lambda}:\widetilde{M}_{\lambda_k}\rightarrow\widetilde{M}_\lambda.
\]
Thus, for all $t\in[0,2\epsilon)$, we can define a smooth family of symplectic forms $\widetilde{\omega_t}=\widetilde{\phi}_{\lambda_k-t}^*(\widetilde{\omega}_{\lambda_k-t})$ on $\widetilde{M}_{\lambda_k}$, so that we can form $\widetilde{\mathcal{J}}(\mathcal{Z}_{\lambda_k},\widetilde{\omega}_t)$.  Notice that given any $\widetilde{J}\in\widetilde{\mathcal{J}}(\mathcal{Z}_k,\widetilde{\omega}_t)$, $\widetilde{J}$ can be pushed forward by $\widetilde{\phi}_{\lambda_k-t}$ to $\widetilde{J}\in\widetilde{\mathcal{J}}(\mathcal{Z}_{\lambda_k-t},\widetilde{\omega}_{\lambda_k-t})$.  

Now choose a $\widetilde{J}\in\widetilde{\mathcal{J}}(\mathcal{Z}_{\lambda_k},\widetilde{\omega}_{\lambda_k})$ so that $\widetilde{J}$ equals $J_0$ near $x_{\lambda_k}$, the point in $M_{\lambda_k}$ being blown up, and so that there is a neighborhood $\mathcal{N}(x_{\lambda_k})$ so that no $\widetilde{J}_\lambda$ holomorphic exceptional divisors intersect $\mathcal{N}(x_{\lambda_k})$.  Then since the taming condition is open, we can choose $\epsilon$ depending on $\widetilde{J}$ small enough so that for all $t\in[0,2\epsilon)$, $\widetilde{J}\in\widetilde{\mathcal{J}}(\mathcal{Z}_{\lambda_k}, \widetilde{\omega}_t)$.  Thus, for each $t\in[0,2\epsilon)$, we can push $\widetilde{J}$ forward by $\widetilde{\phi}_{\lambda_k-t}$ to an almost complex structure $\widetilde{J}(\lambda_k-t)$ to get a family
\[
\widetilde{J}(\lambda)\in\widetilde{\mathcal{J}}(\mathcal{Z}_\lambda,\widetilde{\omega}_\lambda)
\]
for each $\lambda\in(\lambda_k-2\epsilon,\lambda_k]$.  Also, we can choose $\widetilde{J}(\lambda_k)$ so that for all $\lambda\in(\lambda_k-\epsilon,\lambda_k)$
\[
\widetilde{J}(\lambda)\in\widetilde{\mathcal{J}}(\lambda)
\]
where $\widetilde{\mathcal{J}}(\lambda)$ is as before.
Furthermore, as in Remark \ref{almostcomplexstructureremark}, there is a corresponding family 
\[
J(\lambda)\in\mathcal{J}(p_1^\lambda,\ldots,p_n^\lambda;\omega_\lambda)
\]
of almost complex structures on $M_\lambda$ which are integrable near each $p_i^\lambda$ so that $J(\lambda_k)=J_0$ near $x^{\lambda_k}$ and $J(\lambda)\in\mathcal{J}(\lambda)$ for $\lambda\in(\lambda_k-\epsilon,\lambda_k)$. By the above we know that $x^{\lambda_k}$ does not intersect any $J(\lambda)$-holomorphic weighted exceptional divisors.

Now, since for each $t\in(0,2\epsilon)$, $M_{\lambda_k+t}$ is equal to the $(p_k,q_k)$ weighted blowup of $M_{\lambda_k}$ at the point $x_{\lambda_k}$ and $J(\lambda_k)$ equals $J_0$ near $x_{\lambda_k}$, we get corresponding almost complex structures $J_{\lambda_k+t}$ which are integrable near the $(p_k,q_k)$ weighted exceptional divisor.  Also, any orbifold point on $M_{\lambda_k+t}$ either corresponds to some $p_i^{\lambda_k}$ on $M_{\lambda_k}$, or lies on the weighted exceptional divisor.  Thus, $J_{\lambda_k+t}$ is integrable near all the orbifold points $p_i^{\lambda_k+t}$, and we have
\[
J(\lambda_k+t)\in\mathcal{J}(p_1^{\lambda_k+t},\ldots,p_m^{\lambda_k+t};\omega_{\lambda_k+t}).
\]
Also, as before, we can choose $J(\lambda_k)$ so that $J(\lambda)$ defined in this way satisfies $J(\lambda)\in\mathcal{J}(\lambda)$.  Thus, we can blow these almost complex structures up to get almost complex structures $\widetilde{J}(\lambda)\in\widetilde{\mathcal{J}}(\lambda)$ for all $\lambda\in(\lambda_k,\lambda_k+2\epsilon)$.  

In particular, for all $\lambda\in (\lambda_k-2\epsilon,\lambda_k+2\epsilon)\supset I_k$, we have constructed a family $\widetilde{J}(\lambda)$ of almost complex structures on $\widetilde{M}_\lambda$ such that if $\lambda\neq\lambda_k$, $\widetilde{J}(\lambda)\in\widetilde{\mathcal{J}}(\lambda)$.  We define
\[
\widetilde{\mathcal{F}}_{I_k}:=\bigcup_{\lambda\in I_k}\widetilde{J}(\lambda)
\]
Then $\widetilde{\mathcal{F}}_{I_k}$ is diffeomorphic to an open interval, hence it is obviously a topological space.  Furthermore, since for $\lambda\neq\lambda_k$ we have $\widetilde{J}(\lambda)\in\widetilde{\mathcal{J}}(\lambda)$ there is a natural fiberwise inclusion from the piece of $\widetilde{\mathcal{F}}_{I_k}$ over $I_k\cap U_l$ into $\widetilde{\mathcal{F}}_{U_l}$, whenever $I_k\cap U_l\neq\emptyset$.

This completes the construction of a generalized bundle over $S^1$ which we will denote $\widetilde{\mathcal{J}}$.
\end{step}

\begin{step}[Showing that the generalized bundle $\widetilde{\mathcal{J}}$ has a non-zero section.]
We now show that the above generalized bundle has a smooth, non-zero section $\widetilde{J}_\lambda$.  We will do this by taking sections on each $I_j$ and patching them together over the $U_k$.  

First, consider $I_j=(\lambda_j-\epsilon,\lambda_j+\epsilon)$.  Recall from the definition of $\mathcal{F}_{I_j}$ above that for each $\lambda\in(\lambda_j-2\epsilon,\lambda_j+2\epsilon)\supset I_j$, we have an almost complex structure $\widetilde{J}(\lambda)$ on $\widetilde{M}_\lambda$ so that if $\lambda\neq \lambda_k$, $\widetilde{J}(\lambda)\in\widetilde{\mathcal{J}}(\lambda)$.  In particular, this defines a section on $I_j$ which has already been extended a little past $I_j$.

Next consider $U_k$.  We seek to find a section of $\widetilde{\mathcal{J}}$ over $U_k$ which equals $\widetilde{J}(\lambda)$ on $U_k\cap I_j$ whenever this intersection is non-empty.  Fixing a $\lambda_0$, Lemma \ref{homologyinvariantlemma2} gives diffeomorphisms 
\[
\widetilde{\phi}_\lambda:\widetilde{M}_{\lambda_0}\rightarrow\widetilde{M}_\lambda.
\]
We can use these diffeomorphisms to get
\[
\widetilde{\mathcal{J}}(\lambda)\cong\widetilde{\phi}_\lambda^*\widetilde{\mathcal{J}}(\lambda):=\widetilde{\mathcal{J}}(\lambda_0;\lambda)
\]
Notice that any $\widetilde{J}\in\widetilde{\mathcal{J}}(\lambda_0;\lambda)$ is an almost complex structure on $M_{\lambda_0}$ so that 
\[
(\widetilde{\phi}_{\lambda})_*(\widetilde{J})\in\widetilde{\mathcal{J}}(\lambda)
\]
Thus, to find a family $\widetilde{J}(\lambda)\in\widetilde{\mathcal{J}}(\lambda)$ over $U_k$, it suffices to find a path $\widetilde{J}(\lambda)\in\widetilde{\mathcal{J}}(\lambda_0;\lambda)$.

Now, for $\lambda\in(\lambda_k,\lambda_k+2\epsilon)\cup(\lambda_{k+1}-2\epsilon,\lambda_{k+1})$, we already have a choice of $\widetilde{J}(\lambda)$ on $\widetilde{M}_\lambda$, which as above gives us a choice of $\widetilde{J}(\lambda)\in\widetilde{\mathcal{J}}(\lambda_0;\lambda)$.  Consider the interval
\[
U_k\setminus\{(I_k\cup I_{k+1})\cap U_k\}=[\lambda_k+\epsilon,\lambda_{k+1}-\epsilon]
\]
and define the set
\[
\widetilde{\mathcal{J}}_k:=\bigcup_{\lambda\in[\lambda_k+\epsilon,\lambda_{k+1}-\epsilon]}\widetilde{\mathcal{J}}(\lambda_0;\lambda)
\]
This set is obviously fibered over $[\lambda_k+\epsilon,\lambda_{k+1}-\epsilon]$ by $\widetilde{\mathcal{J}}(\lambda_0;\lambda)$, which is a path connected set of $\widetilde{\phi}_\lambda^*(\widetilde{\omega}_\lambda)$-tame almost complex structures on $\widetilde{M}_{\lambda_0}$.  Thus, since the taming condition is open, the set $\widetilde{\mathcal{J}}_k$ defined in this way is path connected.  Also, as pointed out before, we already have two almost complex structures $\widetilde{J}(\lambda_k+\epsilon)$ and $\widetilde{J}(\lambda_{k+1}-\epsilon)$ defined on $\widetilde{\mathcal{J}}_k$, so that we can choose a path $\widetilde{J}(\lambda)$ connecting them so that 
\[
\widetilde{J}(\lambda)\in\widetilde{\mathcal{J}}(\lambda_0;\lambda),
\] which we can then push forward to a family 
\[
\widetilde{J}(\lambda)\in\widetilde{\mathcal{J}}(\lambda)
\]
for all $\lambda\in[\lambda_k+\epsilon,\lambda_{k+1}-\epsilon]$.  

In particular, this gives a path $\widetilde{J}(\lambda)$ on $U_k$ which agrees with the previous choice of $\widetilde{J}(\lambda)$ on $I_j$ whenever $U_k\cap I_j\neq\emptyset$ as desired.
\end{step}

For the rest of the proof, for $\lambda\in S^1$, let $\widetilde{J}(\lambda)$ denote a specific choice of a section of $\widetilde{\mathcal{J}}$, and $J(\lambda)$ the corresponding family of almost complex structures on $M_\lambda$ which pull back to $\widetilde{J}(\lambda)$ under the blowup maps.  To derive a contradiction, we will produce specific exceptional divisors on the spaces $\widetilde{M}_\lambda$ and use $J$-holomorphic curve techniques using the family $\widetilde{J}(\lambda)$ above.

\begin{step}[Constructing exceptional divisors on $U_k$ for odd $k$]
First, assume we have some $\lambda_k$ for $k$ odd which then has isotropy weights $(p_k,q_k,-1)$.  Recall that we have the resolution $\widetilde{M}_\lambda$ with its corresponding set of homology classes 
\[
\mathcal{Z}_\lambda=\bigcup_{i,j}Z_{i,j}^\lambda
\]

Now consider $I_k$.  As in Lemma \ref{local form of (p,q,-r)}, we can choose $\epsilon$ small enough so that the interval $I_{k}=(\lambda_{k}-\epsilon,\lambda_{k}+\epsilon)$ satisfies that given any $\lambda\in(\lambda_k,\lambda_{k}+\epsilon)$, $M_{\lambda}$ is the $(p_k,q_k)$ weighted blowup of size $\lambda-\lambda_{k}$ of $M_{\lambda_k}$ at $x_{\lambda_k}$.  In particular, there is a $(p_k,q_k)$ weighted exceptional divisor in the class $E_k^{\lambda,+}$ which passes through two isolated orbifold singularities of $M_\lambda$.  Thus, as in Remark \ref{weightedblowupremarkfulton} there is an ordering of the classes $Z_{i,j}^\lambda$ from step $1$ and a choice of indices $i^1_{k}=1,m_1$ and $i^2_{k}=1,m_2$ where $Z_{i,j}$ has $i=1,\ldots,m_j$, and a class $\widetilde{E}^{\lambda,+}_k$ satisfying the following properties.
\begin{enumerate}
\item $\widetilde{E}_k^{\lambda,+}$ is an exceptional class in $\widetilde{M}_\lambda$ which is the pullback of $E_k^{\lambda,+}$ under the natural projection from $\widetilde{M}_\lambda\rightarrow M_\lambda$.
\item For $j=1,2$, $\widetilde{E}^{\lambda,+}_k\cdot Z^\lambda_{i^j_{k},1}=1$
\item $\widetilde{E}^{\lambda,+}_k\cdot Z^\lambda_{i,l}=0$ for all other $i,l$.
\end{enumerate}
Furthermore, as $\lambda\in(\lambda_{k},\lambda_{k}+\epsilon)$ increases, $\widetilde{\omega}_\lambda(\widetilde{E}^{\lambda,+}_k)$ also increases, while $\widetilde{\omega}_\lambda(Z^\lambda_{i,j})$ can be fixed to be as small as desired for all $i,j$ and $\lambda$.  

Now recall from Lemma \ref{homologyinvariantlemma} that for some fixed $\lambda_0\in U_{k}$ there are diffeomorphisms $\phi_\lambda$ from $M_{\lambda_0}$ to $M_\lambda$, which by Lemma \ref{diffeomorphismliftinglemma} can be lifted to diffeomorphisms $\widetilde{\phi}_\lambda$ from $\widetilde{M}_{\lambda_0}$ to $\widetilde{M}_\lambda$.  Furthermore, the classes $\widetilde{E}^{\lambda,+}_k$ and $Z^\lambda_{i,j}$ all correspond to $E^{\lambda_0,+}_k$ and $Z^{\lambda_0}_{i,j}$ under these diffeomorphisms.  As such, we will omit the $\lambda$s from the notation, and simply refer to the classes as $\widetilde{E}_k^+$ and $Z^k_{i,j}$.  Furthermore, we can use these diffeomorphisms to extend the classes $\widetilde{E}_k^+$ as being defined over all of $U_{k}$, and we will still have $\widetilde{\omega}_\lambda(E_k^+)$ increases with $\lambda$ while $\widetilde{\omega}_\lambda(Z^k_{i,j})$ can be fixed as small as desired.

Now, consider $\lambda_{k+1}$ which has isotropy weights $(-p_{k+1},-q_{k+1},1)$.  By a similar argument, we can produce classes $\widetilde{E}^-_k$ and indices $i^j_{k}=1,m_j$ satisfying properties $(1)$ and $(2)$ above and so that $\widetilde{\omega}_\lambda(\widetilde{E}^-_k)$ decreases with $\lambda$.
\end{step}

\begin{step}[Deriving a contradiction]
Consider the class $\widetilde{E}^+_1$ as above.  We seek to use $J$-holomorphic curve techniques with the family $\widetilde{J}(\lambda)$ to show that for all $\lambda\in S^1$, the exceptional class $\widetilde{E}^+_1$ has a representative $\widetilde{C}_\lambda^E$ which a smooth, embedded $\widetilde{J}(\lambda)$-holomorphic sphere such that $\widetilde{\omega}_\lambda(\widetilde{C}_\lambda^E)$ is an increasing function of $\lambda$ as $\lambda$ moves counterclockwise around $S^1$.

If this were the case, then picking a base point $\lambda_0\in S^1$ and repeating this argument indefinitely, we would obtain exceptional spheres $\widetilde{C}_{\lambda_0+2\pi i}^E$, one for each $i$.  Also, since $\widetilde{\omega}_\lambda(\widetilde{C}_\lambda^E)$ is an increasing function of $\lambda$, we would have
\[
\widetilde{\omega}_{\lambda_0}(\widetilde{C}^E_{\lambda_0})<\widetilde{\omega}_{\lambda_0}
(\widetilde{C}_{\lambda_0+2\pi}^E)<\ldots<\widetilde{\omega}_{\lambda_0}(\widetilde{C}^E
_{\lambda_0+2k\pi})<\ldots
\]
so that all these exceptional spheres would represent different homology classes,  and the set $\mathcal{E}_{\lambda_0}$ of exceptional classes in $(\widetilde{M}_{\lambda_0},\widetilde{\omega}_{\lambda_0})$ would be infinite where by all our previous assumptions, $\widetilde{M}_{\lambda_0}$ is a closed, symplectic $4$ manifold with $b_2^+>1$, and thus has a finite number of exceptional classes.  This contradiction would then finish the proof of the theorem.

We will establish this by first showing that if $\widetilde{E}_1^+$ is represented by an embedded $\widetilde{J}(\lambda)$-holomorphic sphere $\widetilde{C}_{\lambda'}^E$ in $\widetilde{M}_{\lambda'}$ for some $\lambda'\in U_k$ or $\lambda'\in I_k$, then the same is true for all $\lambda\in U_k$ or $I_k$.  To finish the proof, we will then show that the spheres $\widetilde{C}_{\lambda'}^E$ can be chosen so that $\widetilde{\omega}_\lambda(\widetilde{C}_{\lambda}^E)$ is an increasing function of $\lambda$ as $\lambda$ moves counterclockwise around $S^1$.  

Assume first that for some $\lambda'\in U_k$, $\widetilde{E}_1^+$ is represented by an embedded $\widetilde{J}(\lambda)$ holomorphic sphere $C_{\lambda'}^E$.  By Lemma \ref{homologyinvariantlemma2}, we have diffeomorphisms 
\[
\phi_\lambda:M_{\lambda'}\longrightarrow M_{\lambda},\quad\widetilde{\phi}_\lambda:\widetilde{M}_{\lambda'}\longrightarrow\widetilde{M}
_\lambda.
\]
Thus, if $\widetilde{E}_1^+$ is represented by an embedded, $\widetilde{J}(\lambda')$ holomorphic sphere $\widetilde{C}_{\lambda'}^E$, we can push forward by $\widetilde{\phi}_\lambda$ to obtain embedded, $\widetilde{J}(\lambda)$-holomorphic sphere $\widetilde{C}_\lambda^E$ representing $\widetilde{E}_1^+$, as desired.

Next, assume that for some $\lambda'\in I_k$, $\widetilde{E}_1^+$ is represented by an embedded $\widetilde{J}(\lambda)$ holomorphic sphere $\widetilde{C}_{\lambda'}^E$.  Further assume that $\lambda_k$ has isotropy weights $(p_k,q_k,-1)$.  The case of $(-p_k,-q_k,1)$ has an analagous argument with some sign changes.  Since $\lambda_k$ has weights $(p_k,q_k,-1)$, Lemma \ref{homologyinvariantlemma2} implies that there are diffeomorphisms 
\[
\phi_\lambda:M_\lambda\longrightarrow M_{\lambda_k},\quad \widetilde{\phi}_\lambda:\widetilde{M}_{\lambda}\longrightarrow\widetilde{M}_{\lambda_k}
\]
for all $\lambda\in(\lambda_k-\epsilon,\lambda_k)$.  Thus, pushing forward by $\widetilde{\phi}_\lambda$, we see that $\widetilde{E}_1^+$ is represented by an embedded $\widetilde{J}(\lambda)$ holomorphic sphere $\widetilde{C}_\lambda^E$ in $\widetilde{M}_\lambda$ for all $\lambda\in(\lambda_k-\epsilon,\lambda_k]$ if and only if it is represented by an embedded, $\widetilde{J}(\lambda_k)$-holomorphic sphere $\widetilde{C}_{\lambda_k}^E$ in $\widetilde{M}_{\lambda_k}$.  Thus, to prove that we have spheres $\widetilde{C}_\lambda^E$ as desired for all $\lambda\in I_k$, it suffices to show we have spheres $\widetilde{C}_\lambda^E$ as desired for all $\lambda\in(\lambda_k,\lambda_k+\epsilon)$ if and only if we have a sphere $\widetilde{C}_{\lambda_k}^E$ as desired for $\lambda_k$.  

To see this, first recall from Lemma \ref{local form of (p,q,-r)} that $M_{\lambda_k+\epsilon}$ is the $(p_k,q_k)$-weighted blowup of $M_{\lambda_k}$ at the point $x^{\lambda_k}$, which we recall does not intersect any orbifold points and hence corresponds to a point $\widetilde{x}^{\lambda_k}$ in the resolution $\widetilde{M}_{\lambda_k}$.  In particular, if $C_{\lambda,k}^{E,+}$ is a curve representing $E_k^+$ as a $J(\lambda)$ holomorphic weighted exceptional divisor, there is a map 
\[
\rho_\epsilon:M_{\lambda_k}\longrightarrow M_{\lambda_k+\epsilon}
\]
so that the restriction
\[
\rho_\epsilon:M_{\lambda_k}\setminus\{x^{\lambda_k}\}\longrightarrow M_{\lambda_k+\epsilon}\setminus C_{\lambda,k}^{E,+}
\]
is an orientation preserving diffeomorphism and thus lifts to a diffeomorphism
\[
\widetilde{\rho}_\epsilon:\widetilde{M}_{\lambda_k}\setminus\{\widetilde{x}^{\lambda_k}\}\longrightarrow \widetilde{M}_{\lambda_k+\epsilon}\setminus S_k^{\lambda,+}
\]
where $S_k^+$ is the nodal curve in $\widetilde{M}_{\lambda_0+\epsilon}$ formed by taking the resolution of $C_{\lambda,k}^{E,+}$ as in \ref{weightedblowupremarkfulton}.  This breaks the proof of this case into $2$ subcases.  Namely, if we can define $\widetilde{C}^E_{\lambda_k}$ as desired, we must show that $\widetilde{x}^{\lambda_k}\not\in\widetilde{C}^E_{\lambda_k}$, while if we can define $\widetilde{C}^E_\lambda$ as desired for all $\lambda\in(\lambda_k,\lambda_k+\epsilon)$, we must show that $\widetilde{C}^E_\lambda\cap S_k^{\lambda,+}=\emptyset$.

Consider first the case where we have $\widetilde{C}_{\lambda_k}^E$ defined as desired.  Recall from step $1$ that the almost complex structure $\widetilde{J}(\lambda)$ was chosen so that $\widetilde{x}^{\lambda_k}$ does not intersect any exceptional spheres so that in particular, $\widetilde{x}^{\lambda_k}\not\in\widetilde{C}^E_{\lambda_k}$, as desired.

Next, consider the case where we have $\widetilde{C}_{\lambda}^E$ defined as desired for all $\lambda\in(\lambda_k,\lambda_k+\epsilon)$.  Since $b_2^+(M_\lambda)>1$ for all $\lambda$, Lemma \ref{intersectingweighteddivisorlemma} implies that $\widetilde{C}_{\lambda,k}^{E,+}\cap \widetilde{C}_{\lambda}^E=\emptyset$ and that $\widetilde{C}_i^\lambda\cap\widetilde{C}_{\lambda}^E=\emptyset$, where $\widetilde{C}_i^\lambda$ are representatives of the resolution curves if $C_{\lambda,k}^{E,+}$ as in \ref{weightedblowupremarkfulton}.  Thus, combining these we get $\widetilde{C}^E_\lambda\cap S_k^{\lambda,+}=\emptyset$ as desired.

We now construct the family $\widetilde{C}_\lambda^E$.  Since we assumed $\lambda_1$ has isotropy weights $(p_1,q_1,-1)$, we know from step $3$ that for each $\lambda\in U_1$ there is a $J(\lambda)$ holomorphic weighted exceptional divisor $C_\lambda^E$ in the class $E_1^+$.  Resolving $C_\lambda^E$ as in Remark \ref{weightedblowupremarkfulton} gives in particular an embedded $\widetilde{J}$-holomorphic sphere $\widetilde{C}_\lambda^E$ in the class $\widetilde{E}_1^+$ so that $\widetilde{\omega}_\lambda(\widetilde{C}_\lambda^E)$ increases as $\lambda$ moves counterclockwise around $S^1$.  Then, since $U_1\cap I_2\neq\emptyset$, we can extend this family to $I_2$.  Similarly, $I_2\cap U_2\neq\emptyset$, so we can further extend the family to $U_2$.  A simple induction shows that we can define $\widetilde{C}^E_\lambda$ for all $\lambda\in S^1$.  Furthermore, since it comes from a blowup at $\lambda_1$, we will still have $\widetilde{\omega}_\lambda(\widetilde{C}^E_\lambda)$ increases as $\lambda$ moves counterclockwise around $S^1$, as required.
\end{step}

\end{proof}


\begin{thebibliography}{XXXXXX}

 \bibitem{CHS} Cho, Hwang, and Suh, \emph{Semifree Hamiltonian circle actions on 6-dimensional symplectic manifolds with non-isolated fixed point set}, arXiv:1005.0193v4
 
 \bibitem{DH} Duistermaat, J.J., Heckman, G.: On the variation of cohomology of the symplectic form on the reduced phase space. Invent. Math. 69 (1982), 259-268

 \bibitem{F} W. Fulton, \emph{Introduction to Toric Varieties}, Annals of Math Studies vol 131, PUP (1993).
 
 \bibitem{G} L. Godinho, \emph{Blowing up Symplectic Orbifolds}, Ann. Global Anal. Geom. 20 (2001), 117-62.

 \bibitem{G2} L. Godinho, \emph{On Certain Symplectic Circle Actions}. J. Symplectic Geom. 3 (2005), 357–383.

 \bibitem{D} D. McDuff, \emph{Immersed Spheres in Symplectic $4$-Manifolds}, Annales de l'institut Fourier, 42 no. 1-2 (1992), p. 369-392, doi: 10.5802/aif.1296
 
 \bibitem{D2} D. McDuff, \emph{Nongeneric J-Holomorphic Curves in Rational
4-Manifolds}, arXiv:1211.2431
 
 \bibitem{D3} D. McDuff, \emph{Symplectic Embeddings of $4$-Dimensional Ellipsoids}, J. Topol. \textbf{2} (2009), 1–22.
 
 \bibitem{TW} S. Tolman and J. Weitsman, \emph{On Semifree Symplectic Circle Actions with Isolated Fixed Points}, Topology 39 (2000), no. 2, 299–309.
\end{thebibliography}
\end{document}